\documentclass[onefignum,onetabnum]{siamonline250211}

\usepackage{amsfonts}
\usepackage{graphicx}
\usepackage{epstopdf}
\usepackage{algorithmic}
\ifpdf
  \DeclareGraphicsExtensions{.eps,.pdf,.png,.jpg}
\else
  \DeclareGraphicsExtensions{.eps}
\fi

\usepackage{enumitem}
\setlist[enumerate]{leftmargin=.5in}
\setlist[itemize]{leftmargin=.5in}

\newsiamremark{remark}{Remark}
\newsiamremark{hypothesis}{Hypothesis}
\crefname{hypothesis}{Hypothesis}{Hypotheses}
\newsiamthm{claim}{Claim}

\newcounter{experiment}

\headers{Truncated Differentiation for Inverse MFGs}{S. Liu, Y. T. Chow, and S. Wu Fung}

\title{Truncated Differentiation Through Primal-Dual Solvers for Inverse Potential Mean-Field Games\thanks{\funding{The work of Y. T. Chow was partially supported by NSF DMS-2409903 and ONR N00014-24-1-2661. The work of S. Wu Fung was partially supported by NSF DMS-2309810.}}}

\author{Siting Liu\thanks{Department of Mathematics, University of California, Riverside, Riverside, CA. (\email{sitingl@ucr.edu}, \email{yattinc@ucr.edu})}
\and
Yat Tin Chow\footnotemark[2]
\and
Samy Wu Fung\thanks{Department of Applied Mathematics and Statistics, Colorado School of Mines, Golden, CO.
  (\email{swufung@mines.edu}).}}

\usepackage{amsopn}

\usepackage{mathrsfs}
\newcommand{\Rz}{\mathscr{R}_0}

\ifpdf
\hypersetup{ pdftitle={Truncated Differentiation Through Primal-Dual Solvers for Inverse Potential Mean-Field Games}, pdfauthor={Siting Liu, Yat Tin Chow, and Samy Wu Fung} }
\fi

\usepackage{amsmath}
\usepackage{xcolor}
\usepackage{booktabs}
\usepackage{multirow}
\usepackage{array}
\usepackage{placeins}

\DeclareMathOperator*{\argmin}{arg\,min}
\DeclareMathOperator{\prox}{prox}
\DeclareMathOperator{\spr}{spr}
\newcommand{\cM}{\mathcal{M}}
\newcommand{\cA}{\mathcal{A}}
\newcommand{\cO}{\mathcal{O}}
\newcommand{\cE}{\mathcal{E}}

\newcommand{\pdhgResolventLemma}[1]{
Suppose that $f_\theta$ in~\eqref{eq:disc_objective} is proper,
closed, and convex, that $K$ is linear, and that
Assumption~\ref{asm:pdhg_stepsizes} holds, so that
$\cM\succ0$. Then the dual-extrapolated PDHG
map~\eqref{eq:T_operator_final} coincides with the unit-parameter
metric resolvent of the KKT operator: for every $z=(x,\phi)$, the
update $z^+=T_\theta(z)$ is the unique solution of
\[
\cM(z^+-z)+\cA_\theta(z^+)\ni0,
\]
that is,
\[
T_\theta=R_{\theta,1}=(I+\cM^{-1}\cA_\theta)^{-1}.
\]
}

\theoremstyle{plain}
\newtheorem{assumption}[theorem]{Assumption}

\begin{document}

\maketitle

\begin{abstract}
We study inverse potential mean-field games (MFGs), in which an unknown spatial inverse-cost (mobility) map is inferred from observed population densities.
We solve the forward MFG with a preconditioned primal-dual hybrid gradient (PDHG) method and develop Jacobian-free backpropagation (JFB-$r$), which records only the final $r$ iterations from a detached warm start while retaining the full forward solve.
To analyze this truncated differentiation method, we show that the exact-proximal dual-extrapolated PDHG map is a metric resolvent of the maximal monotone KKT operator.
This resolvent view shows that JFB-$r$ exactly differentiates a finite-trajectory surrogate and, under a locally fixed active set at an exact equilibrium detach point, converges to the implicit gradient as the tracked depth increases.
Across several inverse-MFG settings, numerical experiments show that JFB-r at moderate tracked depths can achieve recovery accuracy comparable to full unrolling while reducing memory and runtime.
\end{abstract}



\section{Introduction}

Mean-field games (MFGs) describe large populations of agents coupled through their distribution~\cite{lasry2007mean,caines2006large}.
Potential MFG equilibria admit convex planning formulations, for which splitting and proximal methods provide scalable solvers~\cite{achdou2012planning,benamou2015augmented,bricenoarias2018proximal,lin2021alternating,jacobs2019solving,liu2021computational,ruthotto2020machine,agrawal2022random}.
Such methods build on augmented-Lagrangian and proximal schemes for dynamic optimal transport~\cite{benamou2000computational,papadakis2014optimal}; see~\cite{liu2021splitting} for splitting methods beyond the potential class.
Related computational challenges arise in high-dimensional optimal control, where neural parameterizations and scalable optimization methods are used to represent feedback policies and interacting-agent dynamics~\cite{onken2022neural,onken2021neural,vidal2025kernel}, and in transport and flow reconstruction, where the governing dynamics are embedded within learning and inverse-problem formulations~\cite{onken2021ot,vidal2023taming,park2026implicit}.
In an \emph{inverse MFG}, the cost governing the agents is unknown and is inferred from observed population densities~\cite{ding2022mean,liu2023inverse}.

The inverse problem requires differentiating the equilibrium solution map.
Existing inverse-MFG approaches include adjoint sensitivity equations~\cite{ren2024policy} and formulations that embed the KKT conditions in a larger system~\cite{chow2023numerical}.
Adjoint sensitivity may require an additional large linearized solve, while embedded KKT formulations enlarge the optimization system.
Backpropagating through every iteration of the forward solver is another direct alternative, but it requires storing the complete solver trajectory.

We instead combine a preconditioned primal-dual hybrid gradient (PDHG) forward solver with Jacobian-Free Backpropagation (JFB)~\cite{fung2022jfb, gelphman2026convergence, yinlearning}. In particular, we consider JFB-$r$, where PDHG runs for $N_{\mathrm{itr}}$ iterations, but only the final $r\ll N_{\mathrm{itr}}$ iterations are recorded.
Thus reverse-mode tape storage and backward cost scale with $r$, up to fixed forward and framework overhead, without shortening the $N_{\mathrm{itr}}$-step forward solve.

Standard JFB guarantees assume a contractive fixed-point map~\cite{fung2022jfb, gelphman2025end, fung2026generalization, mckenzie2024three}, but PDHG is generally only averaged and nonexpansive~\cite{ryu2022large}.
Our analysis instead uses the exact metric-resolvent form of the PDHG map $T_\theta$ defined by the exact proximal update in~\eqref{eq:T_operator_final}, which requires no contraction.
Section~\ref{sec:sensitivity} states the precise scope of these guarantees and their relation to the moderate-depth experiments.

\subsection{Contributions}
\begin{itemize}
\item We develop JFB-$r$ for a preconditioned PDHG solver for inverse potential MFGs, differentiating through only the final $r$ iterations while retaining the full forward solve.

\item We identify the exact-proximal PDHG map $T_\theta$ as an exact metric resolvent and establish directional stability without fixing the flux active set, together with consistency with the implicit gradient under a locally fixed active set and exact equilibrium detachment.
For any differentiable tracked map, JFB-$r$ is also the exact gradient of its finite-trajectory surrogate.

\item We evaluate JFB-$r$ on four inverse-MFG problems, measuring recovery accuracy, memory, and runtime.
Under noise with partial-in-time density observations, JFB-25 reduces peak memory by $3.9\times$ and runtime by half relative to full unrolling, at a recovery error equal within three-seed variability.
\end{itemize}

\section{Setup and Discretization}
\label{sec:setup}

\subsection{Mean-Field Game Formulation}

Let $\rho:\Omega\times[0,T]\to\mathbb R_{\geq0}$ be the population density and $m:\Omega\times[0,T]\to\mathbb R^d$ its flux.
We consider the following potential MFG, whose equilibrium solves the convex planning problem:
\begin{equation}
\inf_{\rho,m} \; f_\theta(\rho,m) \quad \text{s.t.} \quad
\partial_t \rho + \nabla \cdot m = 0,\quad
\rho(s,0) = \mu_0(s),\quad \rho \ge 0,
\label{eq:continuous_mfg}
\end{equation}
with no-flux boundary condition $m\cdot n=0$ and objective
\begin{equation}
f_\theta(\rho,m)
=
\int_0^T\!\int_{\Omega}
\left(
\frac{\|m(s,t)\|^2}{2\,c_\theta(s)\rho(s,t)}
+
\gamma_I \rho(s,t)\log\rho(s,t)
\right) ds\,dt
+
\Psi(\rho(\cdot,T)).
\label{eq:continuous_objective}
\end{equation}
Here $c_\theta>0$ is the inverse-cost (mobility) map, $\gamma_I\geq0$, and
\begin{equation}
\Psi(\rho(\cdot,T))
=
\gamma_T \int_{\Omega} \rho(s,T)
\log\frac{\rho(s,T)}{\mu_1(s)}\,ds,
\qquad \gamma_T > 0.
\label{eq:continuous_terminal}
\end{equation}
We assume that $\mu_0$ and $\mu_1$ have equal mass and that $\mu_1>0$ where the KL term is evaluated.
The kinetic term is the convex perspective cost for $m=\rho v$, and the constraints are linear.
This variational structure descends from the Benamou--Brenier formulation of dynamic optimal transport~\cite{benamou2000computational} and its extension to potential MFG~\cite{benamou2015augmented,benamou2017variational}; finite-difference numerical methods for MFG originate in~\cite{achdou2010mean}, with the planning discretization treated in~\cite{achdou2012planning}; see~\cite{achdou2020mean,meng2025recent} for surveys.

\subsection{Discretized Forward Problem}
\label{subsec:disc}
We discretize the continuous MFG problem~\eqref{eq:continuous_mfg} on a space-time grid of points $(s_i,t_k)$, $i=0,\ldots,N_x-1$, $k=0,\ldots,N_t$, with mesh sizes $h_x$ and $h_t$.
We follow the primal-dual treatment of time-dependent MFG in~\cite{bricenoarias2019implementation}; see~\cite{bricenoarias2018proximal} for the stationary case.
The initial slice $\rho^0\equiv\mu_0$ is fixed.
For notational simplicity, the discrete operators in this subsection are written in one spatial dimension; the multidimensional implementation used in the experiments is the componentwise extension described in appendix~\ref{app:block_pdhg}.
With $\rho:=(\rho^1,\ldots,\rho^{N_t})$ collecting the active density slices and $m:=(m^+,m^-)$ the upwind flux split (with $m^+\ge 0$, $m^-\le 0$), the active primal vector is
\begin{equation}
x:=(\rho,m)\in\mathbb{R}^n,
\qquad
n=N_xN_t+2(N_x-1)N_t,
\label{eq:state_concat}
\end{equation}
where the no-flux boundary slots $m^{+,j}_{N_x-1}$ and $m^{-,j}_0$ are masked out.
The discrete forward objective is written as
\begin{equation}
\begin{aligned}
f_\theta(x)&=\cE_\theta(x)+I_C(x),\\
\cE_\theta(x)&:=
\sum_{j=0}^{N_t-1}\sum_i
\frac{(m^{+,j}_i)^2+(m^{-,j}_i)^2}{2\,c_\theta(s_i)\,\rho^{j+1}_i}
\;+\;
\gamma_I\sum_{k=1}^{N_t}\sum_i\rho^k_i\log\rho^k_i
\;+\;\frac{\Psi(\rho^{N_t})}{h_xh_t}.
\end{aligned}
\label{eq:disc_objective}
\end{equation}
Here $c_\theta$ is the spatial inverse-cost (mobility) map, parameterized by $\theta\in\mathbb{R}^p$ and satisfying $0<c_{\min}\le c_\theta(s)\le c_{\max}$.
The parameter $\gamma_I\ge0$ is the entropy weight, and $\Psi$ is the discrete counterpart of the terminal functional in~\eqref{eq:continuous_terminal}, with weight $\gamma_T>0$ inherited from the continuous form.

The indicator $I_C$ encodes the orthant constraint $C:=\{\rho\ge 0,\,m^+\ge 0,\,m^-\le 0\}$ and is handled by the proximal step in Section~\ref{subsec:PDHG}.
Throughout, the kinetic term is understood as the lower-semicontinuous perspective function: it equals $0$ if $\rho^{j+1}_i=0$ and $m^{\pm,j}_i=0$, and $+\infty$ if $\rho^{j+1}_i=0$ but $m^{\pm,j}_i\neq0$.
We also set $0\log0:=0$; with these conventions, $f_\theta$ is proper, closed, and convex on the orthant.
Under the density-interiority condition in Assumption~\ref{asm:forward_regularity}, this boundary case does not occur.
The smooth part $\cE_\theta$ is twice continuously differentiable wherever the density is bounded away from zero.

The discrete continuity equation is enforced as a single linear constraint $Kx=b$ with
\[
(Kx)^j_i\;=\;\frac{\rho^{j+1}_i-\rho^j_i}{h_t}+(\nabla\cdot m)^j_i,
\qquad j=0,\ldots,N_t-1,
\]
where $(\nabla\cdot m)^j_i$ is the upwind discrete divergence with the no-flux boundary convention; its explicit formula is given in appendix~\ref{app:K_structure}.
Here $\rho^0$ is set to zero inside the operator and its contribution is moved to the right-hand side, so that $b\in\mathbb{R}^{N_xN_t}$ has $b^0_i=\mu_{0,i}/h_t$ and $b^j_i=0$ for $j\ge 1$.
The dual variable $\phi\in\mathbb{R}^{N_c}$ with $N_c=N_xN_t$ is the Lagrange multiplier on the same interval grid as $Kx$.
The explicit form of $K$ and its adjoint, the block dimensions, and the proof that $K$ has full row rank are recorded in appendix~\ref{app:block_pdhg}.
The forward problem is the saddle point
\begin{equation}
\min_{x\in\mathbb{R}^n}\;\sup_{\phi\in\mathbb{R}^{N_c}}\;
f_\theta(x)+\langle\phi,Kx-b\rangle.
\label{eq:disc_forward_problem_compact}
\end{equation}

All subsequent theory concerns the discrete forward and inverse problems on a fixed space–time grid.
The constants in the local assumptions below may depend on $(h_x,h_t)$; no mesh-uniform or discrete-to-continuum limit is claimed.

\subsection{Inverse Mean-Field Game Formulation}\label{sec:inverse}

In the inverse problem, the inverse-cost map $c_\theta(s)$ is unknown, and we estimate its parameters $\theta\in\mathbb{R}^p$ from observed population densities.
Let $\rho^{\mathrm{obs}}$ denote the observations associated with the active density slices $k=1,\ldots,N_t$, and let
\begin{equation}
P_\rho : \mathbb{R}^n \to \mathbb{R}^{N_xN_t},
\end{equation}
denote the linear operator that extracts the active density block $\rho=(\rho^1,\ldots,\rho^{N_t})$ from the primal variable $x=(\rho,m)$, so that $P_\rho x = \rho$.
The fixed initial slice $\rho^0\equiv\mu_0$ is excluded from the comparison since it does not depend on $\theta$.
For a given parameter vector $\theta$, the forward problem~\eqref{eq:disc_forward_problem_compact} produces the equilibrium solution
\begin{equation}
x_\theta^\star = (\rho_\theta^\star, m_\theta^\star).
\end{equation}
If the problem is not globally single-valued, $x_\theta^\star$ denotes the solver-selected branch; the sensitivity statements are local to it.

The inverse MFG problem is therefore
\begin{equation}
\min_{\theta \in \mathbb{R}^p}
\;
F(\theta)
:=
\ell(P_\rho x_\theta^\star,\rho^{\mathrm{obs}}),
\label{eq:inverse_problem}
\end{equation}
where $\ell$ is a data misfit functional, for example
\begin{equation}
\ell(P_\rho x_\theta^\star,\rho^{\mathrm{obs}})
=
\frac12
\|
P_\rho x_\theta^\star - \rho^{\mathrm{obs}}
\|_2^2 .
\label{eq:quadratic_misfit}
\end{equation}
This renders the inverse MFG as training an implicit neural network (or Deep Equilibrium Network)~\cite{el2021implicit, bai2019deep, knutson2026logical}.

The data-misfit functional may include an observation mask, so~\eqref{eq:inverse_problem} also covers partial observations.

With $z=(x,\phi)$, define
\begin{equation}
\widehat\ell(z):=\ell(P_\rho x,\rho^{\mathrm{obs}}),
\label{eq:lifted_loss}
\end{equation}
so that $F(\theta)=\widehat\ell(z_\theta^\star)$ and, whenever $\ell$ is differentiable in its first argument, $\nabla_z\widehat\ell(z)=(P_\rho^\top\nabla_1\ell(P_\rho x,\rho^{\mathrm{obs}}),\,0)$. 
For fixed $\theta$, the forward problem is convex in $(\rho,m)$.
With nonlinear cost parameterizations, such as the neural network used in section~\ref{sec:experiments}, the outer objective $F(\theta)$ is generally nonconvex; the descent statements are therefore local in $\theta$.

\begin{remark}
The parameter-estimation formulation~\eqref{eq:inverse_problem} follows~\cite{ding2022mean}.
Related inverse-MFG approaches include KKT-based formulations~\cite{chow2023numerical}, bilevel optimization~\cite{yu2024bilevel,huang2025joint}, Gaussian-process inference~\cite{guo2024decoding,zhang2025surrogate}, equilibrium-correction iteration~\cite{yu2025equilibrium}, and policy-iteration methods~\cite{ren2024policy}.
Our focus is differentiation through the forward solver; identifiability of the inverse-cost map is studied in~\cite{liu2022inverse,liu2023inverse,imanuvilov2024lipschitz, ren2024reconstructing}.
\end{remark}

\subsection{PDHG Solver}
\label{subsec:PDHG}

We solve~\eqref{eq:disc_forward_problem_compact} using the G-prox preconditioned PDHG method of~\cite{jacobs2019solving,liu2021computational,liu2021splitting,wang2025primal}.
We take the primal and dual step operators $T_x:=\tau I$ and $S:=\sigma W^{-1}$, respectively, where the \emph{exact Schur preconditioner} is
\begin{equation}
W\;:=\;KK^\top\;=\;I_x\otimes B_t\;+\;2L_x\otimes I_t.
\label{eq:W_def}
\end{equation}
Here $L_x$ is the one-dimensional Neumann graph Laplacian, and $B_t=D_tD_t^\top$ is the temporal block (appendix~\ref{app:block_pdhg}).
Since $\rho^0$ is fixed, $D_t$ has full row rank and $W\succ 0$ on the dual space.
We apply $W^{-1}$ using a DCT in space and an eigendecomposition of the small $N_t\times N_t$ time block~\cite{liu2021computational}.

The dual-extrapolated PDHG iteration reads
\begin{align}
\phi^{n+1} &= \phi^n + S(Kx^n-b),
\label{eq:pdhg_dual}\\
x^{n+1} &= \prox_{T_x f_\theta}\!\bigl(x^n - T_xK^\top(2\phi^{n+1}-\phi^n)\bigr),
\label{eq:pdhg_primal}
\end{align}
where $\prox_{T_x f}(y):=\argmin_x\{f(x)+\tfrac12\|x-y\|_{T_x^{-1}}^2\}$.
The proximal step decouples across $(i,j)$ (appendix~\ref{app:block_pdhg}).
Together~\eqref{eq:pdhg_dual} and~\eqref{eq:pdhg_primal} define the fixed-point map $z^{n+1}=T_\theta(z^n)$ with $z=(x,\phi)\in\mathbb{R}^{n+N_c}$:
\begin{equation}
T_\theta(x,\phi)
=
\begin{pmatrix}
\prox_{T_x f_\theta}\!\bigl(x - T_xK^\top(\phi+2S(Kx-b))\bigr) \\
\phi + S(Kx-b)
\end{pmatrix}.
\label{eq:T_operator_final}
\end{equation}
Throughout, $T_\theta$ denotes the exact-proximal map in~\eqref{eq:T_operator_final}, analyzed in sections~\ref{sec:structure} and~\ref{sec:sensitivity}, and $\widetilde T_\theta$ the implemented map obtained by replacing each per-cell proximal solve in~\eqref{eq:T_operator_final} by a fixed budget of $12$ damped Newton iterations (appendix~\ref{app:block_pdhg}).
Algorithm~\ref{alg:jfb_pdhg} and the experiments of section~\ref{sec:experiments} use $\widetilde T_\theta$.

For the exact-proximal analysis, an equilibrium is a fixed point of $T_\theta$:
\begin{equation}
z_\theta^\star
=
T_\theta(z_\theta^\star),
\qquad
z_\theta^\star = (x_\theta^\star,\phi_\theta^\star).
\label{eq:inverse_fixed_point}
\end{equation}
Thus, the inverse problem~\eqref{eq:inverse_problem} can be written as
\begin{equation}
\min_{\theta}
\;
\ell(P_\rho x_\theta^\star,\rho^{\mathrm{obs}})
\quad
\text{subject to}
\quad
z_\theta^\star = T_\theta(z_\theta^\star).
\label{eq:bilevel_inverse}
\end{equation}

\subsection{KKT Operator, Solver Metric, and Step Sizes}
\label{subsec:kkt_metric}

The saddle-point problem~\eqref{eq:disc_forward_problem_compact} is equivalent to the monotone inclusion $0\in\cA_\theta(z)$ for the KKT operator
\begin{equation}
\cA_\theta(x,\phi)
=
\begin{pmatrix}
\partial f_\theta(x)+K^\top\phi\\[1mm]
b-Kx
\end{pmatrix},
\qquad z=(x,\phi).
\label{eq:kkt_operator}
\end{equation}
Its maximal monotonicity is established in section~\ref{sec:structure} (Proposition~\ref{prop:maximal_monotone}).
Only the subdifferential term $\partial f_\theta$ depends on $\theta$; the constraint data $K$ and $b$ do not.
Define the symmetric block matrix
\begin{equation}
\cM =
\begin{pmatrix}
T_x^{-1} & K^\top \\
K & S^{-1}
\end{pmatrix}.
\label{eq:M_metric}
\end{equation}

\begin{assumption}
\label{asm:pdhg_stepsizes}
The step-size operators $T_x=\tau I$ and $S=\sigma W^{-1}$ satisfy
\begin{equation}
\bigl\|S^{1/2} K\, T_x^{1/2}\bigr\|_2^2 < 1.
\label{eq:condat_stepsize}
\end{equation}
Equivalently, since $W=KK^\top$, \eqref{eq:condat_stepsize} is exactly $\tau\sigma<1$ (appendix~\ref{app:stepsize_check}).
\end{assumption}
\begin{remark}
\label{rmk:stepsize}
Inequality~\eqref{eq:condat_stepsize} is the standard PDHG step-size condition~\cite{chambolle2011firstorder,condat2013primal, chambolle2016ergodic}.
It implies $\cM\succ0$, so
\begin{equation*}
\langle z,w\rangle_{\cM}:=\langle \cM z,w\rangle,
\qquad
\|z\|_{\cM}^2:=\langle \cM z,z\rangle
\end{equation*}
define a weighted inner product and norm.
We use $\tau=\sigma=0.99$.
See appendices~\ref{app:stepsize_check} and~\ref{app:resolvent_facts}.
\end{remark}

\subsection{JFB-\texorpdfstring{$r$}{r} Through PDHG}
\label{subsec:jfb_method}

Jacobian-Free Backpropagation (JFB)~\cite{fung2022jfb} runs the full $N_{\mathrm{itr}}$-step PDHG solve but records only the final $r$ steps for reverse-mode differentiation; the state entering those steps is detached. 
JFB has previously been used successfully to train implicit deep learning models~\cite{el2021implicit} for applications in imaging~\cite{heaton2021feasibility}, traffic-flow modeling~\cite{mckenzie2024three}, and arbitrage trading~\cite{heaton2023explainable}.
For $r>1$, JFB-$r$ differentiates the $r$-fold composition $T_\theta^r$; this is closely related to phantom-gradient training~\cite{geng2021training} and truncated bilevel backpropagation~\cite{shaban2019truncated}.
Algorithm~\ref{alg:jfb_pdhg} summarizes the training loop, and section~\ref{sec:sensitivity} characterizes the resulting direction and its relation to the exact implicit gradient.

Let $\bar z$ denote the detached base point, held fixed with respect to $\theta$.
The JFB-$r$ direction at $\theta_0$ is
\begin{equation}
d_{r,\bar z}^{\mathrm{JFB}}(\theta_0)
:=
\Bigl(\partial_\theta\, T_\theta^{\,r}(\bar z)\big|_{\theta=\theta_0}\Bigr)^{\!\top}
\nabla_z\widehat\ell\bigl(T_{\theta_0}^{\,r}(\bar z)\bigr).
\label{eq:jfb_r_def}
\end{equation}
Under the differentiability conditions of Proposition~\ref{prop:jfbr_surrogate}, this is the exact gradient of a finite-trajectory surrogate for any detached $\bar z$.
Comparison with the exact implicit gradient additionally assumes $\bar z=z_{\theta_0}^\star$; under this assumption $T_{\theta_0}^{\,r}(\bar z)=\bar z$, we abbreviate $d_r^{\mathrm{JFB}}:=d_{r,\bar z}^{\mathrm{JFB}}$, and $r=1$ recovers the one-step JFB direction of~\cite{fung2022jfb}.
In computation, Algorithm~\ref{alg:jfb_pdhg} detaches the near-equilibrium warm-up output.
Section~\ref{sec:experiments} reports the detach-point residual.

\begin{algorithm}[htbp]
\caption{JFB-$r$ training through PDHG for inverse MFGs}
\label{alg:jfb_pdhg}
\begin{algorithmic}[1]
\REQUIRE Observations $\rho^{\mathrm{obs}}$, loss $\ell$, initial parameter $\theta^0$, implemented PDHG map $\widetilde T_\theta$, forward steps $N_{\mathrm{itr}}$, tracked depth $1\le r\le N_{\mathrm{itr}}$, outer iterations $N_{\mathrm{out}}$, and optimizer $\mathrm{Opt}$
\ENSURE Learned parameter $\theta^{N_{\mathrm{out}}}$
\FOR{$j=0,1,\ldots,N_{\mathrm{out}}-1$}
    \STATE At $j=0$, initialize $z^0$ from the cold start of appendix~\ref{app:experimental_details}; for $j\ge1$, set $z^0$ to the detached final state of outer step $j-1$.
    \STATE Compute
    $z^{N_{\mathrm{itr}}-r}
    =\widetilde T_{\theta^j}^{\,N_{\mathrm{itr}}-r}(z^0)$
    without recording derivatives, and detach the resulting state
    $\bar z^{\,j}:=z^{N_{\mathrm{itr}}-r}$.
    \FOR{$k=N_{\mathrm{itr}}-r,N_{\mathrm{itr}}-r+1,\ldots,N_{\mathrm{itr}}-1$}
        \STATE Compute $z^{k+1}=\widetilde T_{\theta^j}(z^k)$ with the differentiation tape enabled.
    \ENDFOR
    \STATE Write $z^{N_{\mathrm{itr}}}=(x^{N_{\mathrm{itr}}},\phi^{N_{\mathrm{itr}}})$ and compute $\mathcal L(\theta^j)=\ell(P_\rho x^{N_{\mathrm{itr}}},\rho^{\mathrm{obs}})$.
    \STATE Backpropagate $\mathcal L(\theta^j)$ through the recorded tape to obtain
    $\widehat d_r^{\,j}
    =\nabla_\theta\,\widehat\ell\bigl(\widetilde T_\theta^{\,r}(\bar z^{\,j})\bigr)
    \big|_{\theta=\theta^j}$,
    the exact gradient of the finite-trajectory surrogate of Proposition~\ref{prop:jfbr_surrogate} applied to $\widetilde T_\theta$.
    \STATE Update $\theta^{j+1}=\mathrm{Opt}(\theta^j,\widehat d_r^{\,j})$.
\ENDFOR
\end{algorithmic}
\end{algorithm}

\paragraph{Warm starts}
Each outer step initializes PDHG from the previous step's final state.
The cached state is detached, so gradients do not cross outer iterations, and all $N_{\mathrm{itr}}$ inner steps are still run.

\section{Metric-Resolvent Structure and Directional Stability}
\label{sec:structure}

PDHG is generally nonexpansive rather than contractive, so standard contraction-based analyses of JFB~\cite{fung2022jfb, mckenziedifferentiating, li2026end} do not apply directly.
We instead use the exact metric-resolvent representation of the PDHG map $T_\theta$ defined in~\eqref{eq:T_operator_final}.
Starting from maximal monotonicity of the KKT operator, we establish firm nonexpansiveness, introduce the local curvature needed to eliminate neutral directions, and prove grouped decay of directional state perturbations without assuming strict complementarity for the flux constraints.

\subsection{Metric-Resolvent Form of PDHG}
\label{subsec:resolvent}

We first work under the basic conditions needed for the PDHG resolvent representation: $f_\theta$ is proper, closed, and convex, $K$ is linear, and Assumption~\ref{asm:pdhg_stepsizes} ensures $\cM\succ0$; the local conditions of section~\ref{subsec:forward_geometry} enter only later.

\begin{proposition}[Maximal monotonicity of the KKT operator]
\label{prop:maximal_monotone}
Let $f_\theta=\cE_\theta+I_C$ in~\eqref{eq:disc_objective} be proper, closed, and convex, and let $K$ be linear.
Then the KKT operator $\cA_\theta$ of~\eqref{eq:kkt_operator} is maximal monotone on $\mathbb{R}^{n+N_c}$, and its zeros are exactly the saddle points of~\eqref{eq:disc_forward_problem_compact}.
\end{proposition}

\begin{proof}
Write $\cA_\theta=G_b+\partial\Phi$, where $G_b(x,\phi)=(K^\top\phi,b-Kx)$ is continuous, full-domain, and affine monotone, with skew linear part, and $\Phi(x,\phi)=f_\theta(x)$.
The sum theorem~\cite[Cor.~25.5]{bauschke2017convex} gives maximal monotonicity, while $0\in\cA_\theta(z)$ is precisely the saddle-point KKT system.
\end{proof}

Since $\cA_\theta$ is maximal monotone and $\cM\succ0$, for $\lambda>0$ we define the \emph{metric resolvent}
\begin{equation}
R_{\theta,\lambda}
:=
\bigl(I+\lambda \cM^{-1}\cA_\theta\bigr)^{-1},
\qquad\text{i.e.}\qquad
z^+=R_{\theta,\lambda}(z)
\;\Longleftrightarrow\;
\cM(z^+-z)+\lambda\cA_\theta(z^+)\ni0.
\label{eq:resolvent_def}
\end{equation}
Here $\lambda$ is an auxiliary resolvent parameter, and the PDHG map $T_\theta$ defined in~\eqref{eq:T_operator_final} corresponds to $\lambda=1$.
The interpretation of PDHG as a preconditioned proximal-point iteration is classical~\cite{he2012convergence,chambolle2011firstorder, condat2013primal,vu2013splitting}; the following result identifies the exact resolvent corresponding to our dual-extrapolated, Schur-preconditioned update.

\begin{theorem}[PDHG is a metric resolvent]
\pdhgResolventLemma{main}
\label{thm:pdhg_resolvent}
\end{theorem}

The blockwise verification of the identity is given in appendix~\ref{app:pdhg_resolvent_proof}.

\begin{corollary}[Properties of the metric resolvent]
\label{cor:resolvent_facts}
Let $\cA_\theta$ be maximal monotone on the primal-dual space (Proposition~\ref{prop:maximal_monotone}) and let $\cM\succ0$, $\lambda>0$.
Then:
\begin{enumerate}
\item $R_{\theta,\lambda}$ is everywhere defined and single-valued;
\item $R_{\theta,\lambda}$ is firmly nonexpansive in the $\cM$-inner product:
    \[
    \|R_{\theta,\lambda}z-R_{\theta,\lambda}w\|_{\cM}^2
    \le
    \langle R_{\theta,\lambda}z-R_{\theta,\lambda}w,\;z-w\rangle_{\cM}
    \qquad\text{for all } z,w;
    \]
\item fixed points are exactly equilibria: $R_{\theta,\lambda}(z^\star)=z^\star \iff 0\in\cA_\theta(z^\star)$.
\end{enumerate}
\end{corollary}

Parts (1)--(3) are Minty's theorem and the standard firm-nonexpansiveness computation~\cite{minty1962monotone, rockafellar1976monotone, bauschke2017convex}; a self-contained proof, including $\cM\succ0$, is given in appendix~\ref{app:resolvent_facts}.

Consequently, the map $T_\theta=R_{\theta,1}$ is firmly nonexpansive in the $\cM$-metric and has the MFG equilibria as its fixed points.
Firm nonexpansiveness controls the propagation of state perturbations and gives $\|DR_{\theta,\lambda}\|_{\cM}\le1$ wherever a classical derivative exists, but it provides no contraction factor strictly below one~\cite{ryu2022large}; strict decay relative to an equilibrium requires additional local geometry of the forward objective.

\subsection{Local Curvature of the Forward Problem}
\label{subsec:forward_geometry}

We formalize the required local curvature through the following finite-dimensional conditions, which follow~\cite{yu2024bilevel}.

\begin{assumption}[Regularity of the forward objective]
\label{asm:forward_regularity}
Let $x^\star$ denote the primal component of the fixed point under consideration.
There is a convex neighborhood $U$ of $x^\star$, a bounded parameter neighborhood $\Theta_0$ of the base parameter, and constants $0 < \rho_{\min} \le \rho_{\max} < \infty$, $0 \le m_{\text{max}} < \infty$, and $0 < c_{\min} \le c_{\max} < \infty$ such that the following hold.

\begin{enumerate}
\item The density and flux are locally bounded, with the density bounded away from zero:
    \begin{equation}
    \rho_{\min} \le \rho^k_i \le \rho_{\max},
    \qquad
    \bigl|m^{\pm,j}_i\bigr|\le m_{\text{max}}
    \qquad
    \text{for all } x\in U \text{ and all admissible } i,j,k.
    \end{equation}

\item The spatial inverse-cost map is uniformly positive and bounded:
    \begin{equation}
    0 < c_{\min} \le c_\theta(s_i) \le c_{\max}
    \qquad
    \text{for all grid points } s_i \text{ and all } \theta\in\Theta_0.
    \end{equation}

\item The terminal target is positive, $\mu_{1,i}>0$ for all grid points.

\item The entropy and terminal regularization weights satisfy $\gamma_I > 0$ and $\gamma_T > 0$.
\end{enumerate}
\end{assumption}

On the finite grid, the upper bounds follow locally from continuity after restricting $U$ and $\Theta_0$; the implemented parameterization also supplies a uniform positive lower bound $c_\theta(s_i)\ge0.01$.
The fixed positive vector $\mu_1$ has finite positive lower and upper bounds and does not enter the Hessian estimates.
For the explicit objective~\eqref{eq:disc_objective}, $C^2$ regularity in $(\rho,m)$ follows on $U$ from density interiority and $c_\theta(s_i)>0$.

\begin{lemma}[Smoothness and strong convexity, cf.\ Lemma 3.4 in~\cite{yu2024bilevel}]
\label{lemma:f_theta_scvx}
Under Assumption~\ref{asm:forward_regularity}, the smooth part $\cE_\theta$ is $C^2$, $L_f$-smooth, and $\mu$-strongly convex on $U$, uniformly for $\theta\in\Theta_0$, for some constants $L_f>0$ and $\mu>0$.
Consequently, $f_\theta=\cE_\theta+I_C$ is $\mu$-strongly convex on $U$ in the extended-valued sense.
\end{lemma}

The constants follow from per-cell Hessian bounds and the separability of $\cE_\theta$ across space-time cells; see appendix~\ref{app:scvx_proof}.
In particular, the entropy term both supports density interiority and supplies curvature in the density variables.
This curvature is used below to eliminate nonzero primal kernel directions and may degenerate when $\gamma_I=0$.

\subsection{Directional Stability Without Flux Strict Complementarity}
\label{subsec:nonsmooth}

Although the density remains interior under Assumption~\ref{asm:forward_regularity}, flux variables may lie on the boundary of the orthant constraint.
At a degenerate flux slot, the variable is active while its multiplier vanishes, so the active set need not remain locally constant and the PDHG map need not possess a single classical Jacobian.
We therefore analyze its directional behavior through the graphical derivative of the normal cone.

Fix an equilibrium $z_0=(x_0,\phi_0)$ at $\theta_0$, and define the normal-cone multiplier
\[
v_0:=-\nabla\cE_{\theta_0}(x_0)-K^\top\phi_0\in N_C(x_0),
\]
where $C=\{\rho\ge0,\,m^+\ge0,\,m^-\le0\}$ is the polyhedral orthant product of~\eqref{eq:disc_objective} and $N_C$ its normal-cone map.
Since $C$ is polyhedral, the graphical derivative of $N_C$ at $(x_0,v_0)$ is characterized by the \emph{critical cone}~\cite{rockafellarwets1998, dontchev2014implicit}
\[
\mathcal{K}:=\mathcal{T}_C(x_0)\cap v_0^{\perp}.
\]
Coordinatewise, $\mathcal{K}$ equals $\mathbb{R}$ at inactive slots, $\{0\}$ at strongly active slots, and a half-line at degenerate slots.
Thus $\mathcal{K}$ fails to be a subspace precisely when a degenerate boundary slot is present.

The graphical derivative of the KKT operator $\cA_{\theta_0}$ at $(z_0,0)$ is the set-valued, positively homogeneous operator
\begin{equation}
D\cA(h)
=
\begin{pmatrix}
\nabla^2\cE_{\theta_0}(x_0)\,h_x+K^\top h_\phi+N_{\mathcal{K}}(h_x)\\[1mm]
-K h_x
\end{pmatrix},
\qquad h=(h_x,h_\phi),
\label{eq:DA_def}
\end{equation}
and the \emph{derivative resolvent} is $\Rz:=(I+\cM^{-1}D\cA)^{-1}$.
The theorem below shows that $\Rz$ governs the first-order directional response of one PDHG step at $z_0$.
For $r\ge1$ define
\begin{equation}
q_r:=\sup_{\|h\|_{\cM}=1}\|\Rz^{\,r}(h)\|_{\cM}.
\label{eq:qr_def}
\end{equation}

\begin{theorem}[Directional Stability Without Flux Strict Complementarity]
\label{thm:nonsmooth_kernel}
Suppose Assumptions~\ref{asm:pdhg_stepsizes} and~\ref{asm:forward_regularity} hold at an equilibrium $z_0$ of $\cA_{\theta_0}$, with no assumption on the flux active set.
Then:
\begin{enumerate}
\item $D\cA$ is maximal monotone; $\Rz$ is everywhere defined, single-valued, positively homogeneous, and firmly nonexpansive in $\langle\cdot,\cdot\rangle_{\cM}$, with $\operatorname{Fix}\Rz=\ker D\cA:=\{h:0\in D\cA(h)\}$; and the PDHG map is semi-differentiable (\`{a} la Rockafellar) at $z_0$:
    \[
    T_{\theta_0}(z_0+u)\;=\;z_0+\Rz(u)+o(\|u\|).
    \]
\item  $\ker D\cA=\{0\}$.
\item Consequently $\Rz^n(h)\to0$ for every $h$, and with $q_r$ as in~\eqref{eq:qr_def} one has $q_r\le1$ for every $r$, so the (M)-norm does not increase at any depth, and there is a finite $r_*\ge1$ with
    \[
    q_{r_*}<1,
    \qquad
    \|\Rz^n(h)\|_{\cM}\le q_{r_*}^{\lfloor n/r_*\rfloor}\,\|h\|_{\cM}
    \quad\text{for all } n\ge0.
    \]
\end{enumerate}
\end{theorem}

Part~(3) gives a strict norm reduction relative to the fixed point after finite grouping; no Lipschitz contraction estimate between arbitrary inputs follows.

To prove that the kernel is trivial, let $w\in N_{\mathcal{K}}(h_x)$ be the normal-cone term in the primal kernel equation and pair that equation with $h_x$.
The equality constraint gives $Kh_x=0$, while the cone relation gives $\langle w,h_x\rangle=0$; strong convexity therefore forces $h_x=0$.
Density interiority then makes the density slots of the critical cone unconstrained, so the density block of the remaining kernel equation gives $D_t^\top h_\phi=0$.
Since $D_t^\top$ is injective, $h_\phi=0$.
Firm nonexpansiveness and the trivial fixed-point set yield pointwise decay of $\Rz^n h$; positive homogeneity, continuity, and compactness of the unit sphere upgrade this to a uniform strict decrease after a finite number of iterations.
The complete proof is given in appendix~\ref{app:nonsmooth_kernel}.

At a flux active-set kink, $\Rz$ is generally piecewise linear rather than linear, so the theorem does not provide a classical parameter-sensitivity formula.
The result is local to the equilibrium and concerns iterations of the derivative map; the nonlinear PDHG trajectory need not remain inside $U$.

\section{Local Sensitivity and JFB Consistency}
\label{sec:sensitivity}

Section~\ref{sec:structure} established directional stability without requiring classical differentiability.
To compare JFB-$r$ with implicit differentiation, we now assume that the flux active set is locally fixed near a base equilibrium, so the reduced KKT system and the tracked PDHG steps are $C^1$ there.
We first identify the finite-trajectory surrogate differentiated by JFB-$r$, and then show that, under exact equilibrium detachment, its sensitivity converges to the exact implicit gradient as the tracked depth increases.
Throughout, gradients and descent pairings in $\theta$ use the Euclidean inner product on the parameter space; the metric $\cM$ acts only on the primal-dual space.

\subsection{Differentiability Under a Locally Fixed Active Set}
\label{subsec:smooth_bridge}

\begin{assumption}[Locally fixed active set]
\label{asm:local_chart}
Fix a base parameter $\theta_0$ and an equilibrium $z_0=(x_0,\phi_0)$ satisfying $0\in\cA_{\theta_0}(z_0)$, with $x_0$ in the neighborhood $U$ of Assumption~\ref{asm:forward_regularity}.
Assume that there exist a neighborhood of $(\theta_0,z_0)$ and a fixed set $\mathcal A$ of active flux coordinates such that, along the local KKT and proximal branches through $(\theta_0,z_0)$, the coordinates in $\mathcal A$ remain active with strictly complementary multipliers, while the remaining flux coordinates remain inactive.
After eliminating the coordinates in $\mathcal A$ and reusing $z$ for the retained variables, the local KKT conditions are equivalent to a $C^1$ reduced residual equation $\mathcal F_{\mathcal A}(\theta,z)=0$, formed by the retained stationarity equations together with all equality constraints.
\end{assumption}

The density coordinates are retained because they are bounded away from the orthant boundary on $U$.
Assumption~\ref{asm:local_chart} constrains only the local KKT structure; the data misfit enters separately.

Assumption~\ref{asm:local_chart} is restrictive for the flux variables.
In the upwind split used here, many flux components lie on the orthant boundary.
At strongly active slots, strict complementarity keeps that slot on a locally constant branch, on which the per-cell prox is $C^1$; at degenerate boundary slots the flux prox is not classically differentiable, and autodiff follows the generalized-derivative convention selected by the implementation.
The diagnostics in section~\ref{sec:experiments} compare directions produced by that same convention.
Across active-set changes we rely only on the stability dichotomy of Theorem~\ref{thm:nonsmooth_kernel}; a Bouligand/Clarke theory for such changes is beyond our scope.

In the reduced coordinates, let $K$ denote the retained constraint matrix and set
\begin{equation}
H:=D_z\mathcal F_{\mathcal A}(\theta_0,z_0)
=
\begin{pmatrix}
Q & K^\top\\
-K & 0
\end{pmatrix},
\qquad
Q:=\nabla^2\cE_{\theta_0}(x_0),
\qquad
B:=D_\theta\mathcal F_{\mathcal A}(\theta_0,z_0).
\label{eq:H_block}
\end{equation}
Because the density variables are never eliminated, the reduced matrix $K$ retains the temporal block $D_t$ and therefore has full row rank (appendix~\ref{app:K_structure}).
We continue to denote the reduced metric by $\cM$; the projection notation is suppressed and recorded in appendix~\ref{app:consistency_proof}.

\begin{lemma}[Reduced KKT invertibility and sensitivity]
\label{lem:H_invertible}
Under Assumptions~\ref{asm:forward_regularity} and~\ref{asm:local_chart}, the matrix $H$ of~\eqref{eq:H_block} is monotone and invertible.
Consequently, within the fixed-active-set regime of Assumption~\ref{asm:local_chart}, the equilibrium admits a locally unique $C^1$ branch $\theta\mapsto z_\theta^\star$ near $\theta_0$, with
\begin{equation}
D_\theta z_{\theta_0}^\star=-H^{-1}B.
\label{eq:true_sensitivity}
\end{equation}
\end{lemma}

\begin{proof}
\[
H+H^\top=
\begin{pmatrix}
2Q&0\\
0&0
\end{pmatrix}\succeq0 .
\]
If $H(h_x,h_\phi)=0$, then $Kh_x=0$, and pairing the first block equation with $h_x$ gives $h_x^\top Qh_x=0$.
The positive definiteness of $Q$ (Lemma~\ref{lemma:f_theta_scvx}) yields $h_x=0$, after which the full row rank of $K$ gives $h_\phi=0$.
The implicit function theorem then gives the stated local branch and sensitivity formula.
Strict complementarity and continuity keep this local solution branch in the same KKT regime, so it solves the full KKT system and not only the reduced residual equation.
\end{proof}

In particular, multiplier uniqueness follows from the full row rank of $K$ and need not be assumed.

Under Assumption~\ref{asm:local_chart}, the critical cone from section~\ref{subsec:nonsmooth} becomes the retained-coordinate subspace.
Consequently the directional KKT operator and derivative resolvent reduce, in these coordinates, to the classical linearizations used below.
The formal coordinate reduction is given in appendix~\ref{app:consistency_proof}.

\subsection{Surrogate and Implicit Gradients}
\label{subsec:exact_gradient}

We first consider an arbitrary detached base point $\bar z$, treated as independent of $\theta$ during differentiation.
For $r\ge1$, define the finite-trajectory surrogate
\begin{equation}
G_{r,\bar z}(\theta)
:=
\widehat\ell\bigl(T_\theta^{\,r}(\bar z)\bigr).
\label{eq:surrogate_r}
\end{equation}

\paragraph{Differentiable data misfit}
Wherever a loss gradient appears below, we assume in addition that $\ell$ is continuously differentiable at the predicted densities at which it is evaluated.
In particular, near the base equilibrium we assume that $\ell$ is continuously differentiable near $(P_\rho x_0,\rho^{\mathrm{obs}})$, so that the lifted loss $\widehat\ell$ of~\eqref{eq:lifted_loss} is $C^1$ near $z_0$, and we set $g:=\nabla_z\widehat\ell(z_0)$.
For the quadratic misfit in~\eqref{eq:quadratic_misfit}, this holds.

\begin{proposition}[JFB-$r$ as a finite-trajectory surrogate gradient]
\label{prop:jfbr_surrogate}
Suppose that $\theta\mapsto T_\theta^{\,r}(\bar z)$ is differentiable at $\theta_0$ and $\widehat\ell$ is differentiable at $T_{\theta_0}^{\,r}(\bar z)$.
Then
\begin{equation}
\nabla_\theta G_{r,\bar z}(\theta_0)
=
\Bigl(\partial_\theta T_\theta^{\,r}(\bar z)\big|_{\theta=\theta_0}\Bigr)^{\!\top}
\nabla_z\widehat\ell\bigl(T_{\theta_0}^{\,r}(\bar z)\bigr)
=
d_{r,\bar z}^{\mathrm{JFB}}(\theta_0).
\label{eq:surrogate_grad}
\end{equation}
Thus, whenever it is nonzero, $-d_{r,\bar z}^{\mathrm{JFB}}$ is a strict descent direction for $G_{r,\bar z}$.
\end{proposition}

\begin{proof}
This follows directly from the chain rule with $\bar z$ held fixed.
\end{proof}

The identity and its proof hold verbatim for any differentiable parameterized map in place of $T_\theta$, in particular for the implemented $\widetilde T_\theta$ of Algorithm~\ref{alg:jfb_pdhg}; the experiments use it in this form.

Assumption~\ref{asm:local_chart} supplies a sufficient condition for this differentiability.
If all $r$ tracked proximal solves remain in the fixed-active-set neighborhood \emph{and} select the same active set $\mathcal A$, each step satisfies
\[
\cM(z^+-z)+\mathcal F_{\mathcal A}(\theta,z^+)=0 .
\]
Its derivative with respect to $z^+$ has positive-definite symmetric part, so the implicit function theorem makes each tracked step, and hence their composition, $C^1$.
A tracked solve that lands near $z_0$ but selects a different active set follows a different local branch, so proximity alone does not guarantee that the fixed-active-set argument applies.

Detaching $\bar z$ removes its dependence on $\theta$, whereas the equilibrium state $z_\theta^\star$ varies implicitly with the parameter, so descent for $G_{r,\bar z}$ does not transfer to the true inverse objective $F$.
We next compare these two sensitivities at an equilibrium.

By Lemma~\ref{lem:H_invertible} and the chain rule, the exact equilibrium gradient is
\begin{equation}
\nabla_\theta F(\theta_0)
=
\bigl(-H^{-1}B\bigr)^\top g
=
-B^\top H^{-\top}g .
\label{eq:implicit_gradient}
\end{equation}

\subsection{Sensitivity Recurrence and Large-Depth Consistency}
\label{subsec:consistency}

To compare JFB-$r$ directly with the implicit gradient, we now take the detached base point to be the exact equilibrium, $\bar z=z_0=z_{\theta_0}^\star$.
Since $T_{\theta_0}(z_0)=z_0$ (Corollary~\ref{cor:resolvent_facts}(3)), every tracked iterate at the base parameter equals $z_0$, and the same state and parameter linearizations govern all $r$ steps.

In the reduced coordinates, define
\begin{equation}
L:=(\cM+H)^{-1}\cM,
\qquad
P:=-(\cM+H)^{-1}B .
\label{eq:LP_def}
\end{equation}
The matrices are well defined because the symmetric part of $\cM+H$ is bounded below by $\cM\succ0$.
Moreover
\[
L=D_zT_{\theta_0}(z_0),
\qquad
P=\left.\partial_\theta T_\theta(z_0)\right|_{\theta=\theta_0},
\]
so $L$ is the classical fixed-active-set specialization of the derivative resolvent from section~\ref{sec:structure}.

\begin{proposition}[Sensitivity recurrence at an exact detach point]
\label{prop:side_by_side}
Under Assumptions~\ref{asm:pdhg_stepsizes}, \ref{asm:forward_regularity}, and~\ref{asm:local_chart} at $(\theta_0,z_0)$ with $\bar z=z_0=z_{\theta_0}^\star$, the tracked sensitivities $u_k:=\left.\partial_\theta T_\theta^{\,k}(z_0)\right|_{\theta=\theta_0}$ satisfy
\[
u_0=0,
\qquad
u_{k+1}=Lu_k+P,
\]
and hence
\begin{equation}
u_r=\sum_{k=0}^{r-1}L^kP,
\qquad
d_r^{\mathrm{JFB}}
=
P^\top\sum_{k=0}^{r-1}(L^\top)^k g .
\label{eq:jfb_closed_form}
\end{equation}
In addition,
\[
\operatorname{Fix}(L)=\ker H=\{0\},
\qquad
\spr(L)<1 .
\]
\end{proposition}

The recurrence follows by differentiating each tracked step at the fixed point; it concerns first-order sensitivities and does not describe the nonlinear PDHG iterates.
The spectral conclusion follows from the monotonicity and invertibility of $H$; the matrix proof is given in appendix~\ref{app:consistency_proof}.

Equivalently, let $w_0=0$ and define
\begin{equation}
\cM(w_{k+1}-w_k)+H^\top w_{k+1}-g=0 .
\label{eq:adjoint_pp}
\end{equation}
Then
\[
d_r^{\mathrm{JFB}}=-B^\top w_r,
\qquad
w_r\longrightarrow H^{-\top}g .
\]
Thus, under exact equilibrium detachment, JFB-$r$ can be interpreted as performing $r$ metric-preconditioned iterations toward the implicit adjoint, without explicitly solving the adjoint system.
A quantitative alignment statement for these adjoint iterates, together with the identity isolating the final compression through $BB^\top$, is given in appendix~\ref{app:adjoint_alignment}; $\cM$-alignment of the adjoints need not transfer to the parameter gradients.

\begin{theorem}[Large-depth consistency under exact equilibrium detachment]
\label{thm:jfb_consistency}
Under Assumptions~\ref{asm:pdhg_stepsizes}, \ref{asm:forward_regularity}, and~\ref{asm:local_chart} at $(\theta_0,z_0)$ and exact equilibrium detachment $\bar z=z_0=z_{\theta_0}^\star$,
\[
d_r^{\mathrm{JFB}}(\theta_0)
\longrightarrow
\nabla_\theta F(\theta_0)
\qquad\text{as }r\to\infty .
\]
More precisely,
\[
d_r^{\mathrm{JFB}}-\nabla_\theta F(\theta_0)
=
-P^\top(L^\top)^r(I-L^\top)^{-1}g .
\]
If $\nabla_\theta F(\theta_0)\neq0$, then there exists $r_0<\infty$ such that $-d_r^{\mathrm{JFB}}$ is a strict descent direction for $F$ for every $r\ge r_0$.
\end{theorem}

The proof uses
\[
\sum_{k=0}^{\infty}L^kP
=
(I-L)^{-1}P
=
-H^{-1}B
\]
and is given in appendix~\ref{app:consistency_proof}.

Because $L$ may be nonnormal, $\spr(L)<1$ controls only the asymptotic decay and gives no useful bound at a prescribed moderate depth.
Theorem~\ref{thm:nonsmooth_kernel}(3) guarantees a finite $r_*$ with $\|L^{\,r_*}\|_{\cM}<1$ on the invariant subspace of Proposition~\ref{prop:smooth_bridge}, but no explicit value of $(r_*,\|L^{\,r_*}\|_{\cM})$ is available a priori; the resulting finite-depth bias bound is recorded in appendix~\ref{app:consistency_proof}.

The consistency result is local: it assumes a locally fixed active set, the exact-proximal map $T_\theta$, and exact equilibrium detachment, whereas Algorithm~\ref{alg:jfb_pdhg} runs the finite-Newton map $\widetilde T_\theta$ from a warm detached near-equilibrium state.
Section~\ref{sec:experiments} therefore evaluates this moderate-depth, warm-detach regime through branch-consistent gradient diagnostics.

\section{Numerical Experiments}\label{sec:experiments}

We test JFB on four inverse MFG problems.
Experiment~1 measures how tracked depth affects gradient quality and inverse recovery (section~\ref{subsec:exp_gradient}); Experiment~2 compares runtime and peak memory with full unrolling and checkpointed AD as the discretization grows (section~\ref{subsec:exp_scaling}); Experiments~3 and~4 test recovery under noisy, partial-in-time, and multi-instance observations (section~\ref{subsec:exp_recovery}).

\subsection*{Common setup}

All reconstruction runs use the PDHG solver
\eqref{eq:pdhg_dual}--\eqref{eq:pdhg_primal}, the exact Schur
preconditioner $W=KK^\top$, $\tau=\sigma=0.99$, and entropy weight
$\gamma_I=0.01$. All experiments use the time horizon $T=1$, so
$h_t=1/N_t$; the experiment-specific terminal weight $\gamma_T$ is
listed in Table~\ref{tab:experimental_configs}. Each outer iteration
warm-starts from the previous
detached primal-dual state. We parameterize
\[
c_\theta(s)=0.01+\operatorname{softplus}(\operatorname{MLP}_\theta(s))
\]
with a three-layer tanh MLP and train with Adam. For a binary observation mask $\cO$, let $|\cO|$ denote the number of observed entries and $\odot$ the entrywise product; the finite-solve training loss is
\[
\mathcal L(\theta)
=|\cO|^{-1}\|\cO\odot(\rho_\theta-\rho^{\mathrm{obs}})\|_2^2.
\]
Write $\rho_\theta:=P_\rho x^{N_{\mathrm{itr}}}$ for the density block of the finite-solve output; the detached base point $\bar z$ of section~\ref{subsec:jfb_method} is here the warm state entering the final $r$ steps.
Throughout, $\rho^{\mathrm{obs}}$ denotes the fitted data,
which may be noisy and partially observed.
We compare against full-trajectory automatic differentiation (AD,
i.e.\ full unrolling of all $N_{\mathrm{itr}}$ steps), its
rematerialized variant (AD-ckpt), and implicit differentiation (ID)
through the fixed-point equation.
AD differentiates the full $N_{\mathrm{itr}}$-step solve from the incoming warm state; JFB-$r$ differentiates only the final $r$ steps from the detached state $\bar z$. Both use the implemented update
$\widetilde T_\theta$, whose cellwise proximal problem is approximated
by damped Newton steps, rather than the exact-proximal map $T_\theta$
analyzed in section~\ref{sec:sensitivity}. ID linearizes the
fixed-point equation of $\widetilde T_\theta$ at the returned state;
when that state is close to a fixed point and the selected branch is
differentiable, this is an implementation-level implicit reference, and
otherwise a branch-linearized diagnostic.
We report the grid-weighted relative inverse-cost-map error
\[
\operatorname{RelErr}(c_\theta,c^\star)
=100\,\frac{\|c_\theta-c^\star\|_{2,h_x}}
{\|c^\star\|_{2,h_x}}.
\]
Where reported, the best-by-loss error for each seed is the
reconstruction error at the iterate minimizing that seed's recorded
training loss; best-by-loss selection is tracked at every outer
iteration (the stored parameter history is subsampled for storage
only). Captions distinguish best-by-loss
from final-iterate error. Time
and memory
entries reported without a spread are means over the seeds of the
experiment, except where a caption states a different convention.
Unless stated otherwise, results use three seeds and one NVIDIA
RTX~6000 Ada GPU in double precision. All seeds and per-run
configurations are fixed in the released repository. Exact optimizer, architecture,
data-generation, masking, seed, and memory-measurement details are in
appendix~\ref{app:experimental_details}.

\refstepcounter{experiment}\label{sec:exp1}
\subsection{Gradient Validation and Tracked Depth
(Experiment~\theexperiment)}
\label{subsec:exp_gradient}

We use $\Omega=[-0.5,0.5]$, $N_x=64$, $N_t=16$, and
$c^\star(s)=0.7-0.3\cos(2\pi s)$, with a full clean density
trajectory. All methods use a width-32 MLP, $500$ outer iterations,
and $N_{\mathrm{itr}}=100$ PDHG steps per outer iteration. AD
differentiates all $100$ steps; AD-ckpt computes the same direction
by rematerialization; ID uses the matched detached warm state; JFB-$r$
differentiates only the final $r$ steps. Numerical solver settings are in
appendix~\ref{app:exp1_details}.

\begin{table}[htbp]
\centering
\caption{Experiment 1 comparison (three seeds). All methods use the
matched warm-started forward protocol; ID is the matched warm-start
implicit method. Errors are final-iterate mean $\pm$ standard
deviation.}
\label{tab:exp1}
\begin{tabular}{lcccc}
\toprule
{Method} & {Rel.\ Error (\%)} & {Total Time (s)} & {Peak Mem.\ (MB)} & {Speedup} \\
\midrule
AD       & $0.95 \pm 0.09$ & $32.3 \pm 0.0$   & $142.6$ & $1.0\times$ \\
AD-ckpt  & $0.94 \pm 0.10$ & $38.7 \pm 0.1$   & $64.1$  & $0.83\times$ \\
ID       & $0.94 \pm 0.09$ & $69.4 \pm 0.6$   & $64.1$  & $0.47\times$ \\
JFB-10   & $1.13 \pm 0.15$ & $13.2 \pm 0.0$   & $48.1$  & $\mathbf{2.4\times}$ \\
JFB-5    & $2.51 \pm 0.17$ & $12.1 \pm 0.1$   & $48.1$  & $\mathbf{2.7\times}$ \\
\bottomrule
\end{tabular}
\end{table}

As Table~\ref{tab:exp1} shows, AD, AD-ckpt, and ID reach
$0.94$--$0.95\%$ error once their forward protocols are matched.
ID saves memory at the cost of backward Krylov solves.
JFB-10 gives up about $0.2$ percentage points of accuracy for a $2.4\times$ speedup and about one third of AD's memory. All methods
run the identical forward solve (precision, $N_{\mathrm{itr}}$, and
warm-start rule); the systematic timing differences come from
differentiation.

Under the tested settings $r=1$ fails, whereas $r=5$ and $r=10$
improve recovery; the depth ablation and its learning-rate dependence
are in appendix~\ref{app:exp1_ablation}. This matches the surrogate interpretation of Proposition~\ref{prop:jfbr_surrogate}: wherever the tracked trajectory is differentiable, JFB-1 differentiates a one-step surrogate whose gradient need not align with the true objective.

Along three JFB-10 training runs, we evaluate $d^{\mathrm{JFB}}_{r}$
against a separately cold-solved implicit reference
$d^{\mathrm{ID\text{-}ref}}$ at $63$ checkpoints and
$r\in\{1,5,10,20\}$. Under a locally fixed active set the reference is
the implicit gradient in~\eqref{eq:implicit_gradient}; at kinked flux
entries it is the one-sided branch selected by the implementation. It serves as a
diagnostic reference, separate from the ID training method of
Table~\ref{tab:exp1}. Full protocol and reference checks are in
appendix~\ref{app:exp1_alignment_details}.

\begin{table}[htbp]
\centering
\small
\setlength{\tabcolsep}{3pt}
\caption{Experiment 1 descent-phase alignment diagnostics. The descent
phase is the $15$ checkpoints at outer iterations at most $100$ across
the three JFB-10 trajectories. Entries are
$\cos(d^{\mathrm{JFB}}_{r},d^{\mathrm{ID\text{-}ref}})$ and the
relative bias
$\|d^{\mathrm{JFB}}_{r}-d^{\mathrm{ID\text{-}ref}}\|/
\|d^{\mathrm{ID\text{-}ref}}\|$. Minimum cosines and statistics over
all $63$ checkpoints are in
appendix~\ref{app:exp1_alignment_details}.}
\label{tab:exp1_alignment}
\begin{tabular}{rrrr}
\toprule
$r$ & median $\cos$ & positive checkpoints & med.\ rel.\ bias \\
\midrule
$1$  & $-0.62$ & $2/15$  & $1.00$ \\
$5$  & $+0.75$ & $14/15$ & $0.90$ \\
$10$ & $+0.96$ & $15/15$ & $0.64$ \\
$20$ & $+0.90$ & $15/15$ & $0.60$ \\
\bottomrule
\end{tabular}
\end{table}

At the states visited by JFB-10 during descent, the truncated directions are positively aligned with the implementation-level implicit reference at $14/15$ checkpoints for $r=5$ and $15/15$ for $r\ge10$, although the median relative bias stays of order one at every depth.
JFB-1 is positively aligned at only $2$ of $15$ descent-phase checkpoints, matching its training failure. Because the $r\neq10$ directions are evaluated retrospectively
at JFB-10 states, the table compares alignment at a common set of
states; recovery along each depth's own trajectory is reported in
Table~\ref{tab:exp1} and appendix~\ref{app:exp1_ablation}.

At the true-cost equilibrium, density remains interior, but strict
complementarity fails for some boundary flux coordinates; active sets
also change during training. Theorem~\ref{thm:jfb_consistency}
therefore does not apply directly, and the reported alignment results
are branch-linearized diagnostics. The small detach residual, whose
median is $2\times10^{-7}$ (maximum $1.2\times10^{-5}$, excluding the
initial outer iteration, which begins without a cached warm state),
indicates proximity to a fixed point but does not
bound the sensitivity error. Further active-set and detach-residual
diagnostics are in appendix~\ref{app:extended_diag}, and the full
alignment statistics are in
appendix~\ref{app:exp1_alignment_details}.

\refstepcounter{experiment}\label{sec:exp2}
\subsection{Two-Dimensional Computational Scaling
(Experiment~\theexperiment)}
\label{subsec:exp_scaling}

We use
$c^\star(s_1,s_2)=0.7-0.15\cos(2\pi s_1)-0.15\cos(2\pi s_2)$
and refine $32^2$, $64^2$, and $128^2$ spatial grids with
$N_t=16,32,64$, respectively. A width-64 MLP is trained for $300$
outer iterations with $N_{\mathrm{itr}}=100$. The comparison is
primarily computational; error entries at $32^2$ and $64^2$
aggregate three seeds, the resource-bound $128^2$ runs use one seed,
and forward-budget checks are in appendix~\ref{app:exp2_checks}.

\begin{table}[htbp]
\centering
\caption{Experiment 2 scaling on one 48~GB RTX~6000 Ada
($300$ outer iterations, warm-started $N_{\mathrm{itr}}=100$).
Errors at $32^2$ and $64^2$ are the median (min--max) of three
seeds at the final iterate (appendix~\ref{app:exp2_checks}); the
$128^2$ rows are single resource-bound runs. Time and memory are
from the first seed.}
\label{tab:scaling_results}
\begin{tabular}{llrrr}
\toprule
{Grid} & {Method} & {Rel.\ Error (\%)} & {Total Time (s)} & {Peak Mem.\ (GB)} \\
\midrule
$32\times 32$   & AD      & $1.45$ {\footnotesize$(1.41$--$1.48)$} &   $33$  &  $2.51$ \\
                & JFB-10  & $2.01$ {\footnotesize$(1.92$--$2.03)$} &   $16$  &  $0.27$ \\
\midrule
$64\times 64$   & AD      & $1.63$ {\footnotesize$(1.43$--$4.95)$} &  $102$  & $20.07$ \\
                & JFB-10  & $2.36$ {\footnotesize$(2.30$--$72.2)$} &   $33$  &  $2.19$ \\
\midrule
$128\times 128$ & AD      & \multicolumn{3}{c}{\textit{out of memory ($172$~GB allocation)}} \\
                & AD-ckpt & $1.48$ &  $927$  & $16.28$ \\
                & JFB-10  & $2.92$ &  $179$  & $17.48$ \\
\bottomrule
\end{tabular}
\end{table}

In Table~\ref{tab:scaling_results}, JFB-10 reduces both runtime and
peak memory relative to full AD at $32^2$ and $64^2$. At $128^2$ plain
AD exceeds the 48~GB GPU memory, while checkpointed AD completes the
same training in $16.28$~GB and $927$~s. JFB-10 is then about
$5.2\times$ faster than checkpointed AD at comparable peak memory
($17.48$~GB against $16.28$~GB), with single-run recovery errors of
$2.92\%$ and $1.48\%$.

Some $64^2$ runs become unstable after the data misfit has already
saturated; Table~\ref{tab:scaling_results} reports the resulting
ranges, with individual trajectories and forward-budget checks in
appendix~\ref{app:exp2_checks}.

\subsection{Recovery under Noise, Partial-in-Time Observations, and
Multiple Instances}
\label{subsec:exp_recovery}

\refstepcounter{experiment}\label{sec:exp3}
\paragraph{Single-instance recovery (Experiment~\theexperiment)}

Experiment~3 reuses Experiment~2's ground-truth map
$c^\star(s_1,s_2)=0.7-0.15\cos(2\pi s_1)-0.15\cos(2\pi s_2)$ on the
same $32^2$ grid with $N_t=16$, now under degraded observations and a
longer budget of $1000$ outer iterations. We compare four observation
regimes:
clean data with all output times observed; noise $\gamma_n=0.01$
(additive uniform noise of width $\gamma_n$, i.e.\
$\gamma_n\,U[-1/2,1/2]$; see appendix~\ref{app:exp3_details}) with
all output times observed; clean data with $50\%$ of output times
observed; and noise $\gamma_n=0.05$ with $50\%$ of output times
observed, i.e.\ partial-in-time density observations. A width-64 MLP
is trained with $N_{\mathrm{itr}}=100$ PDHG steps per outer iteration.
JFB-25 is
included in the hardest setting. Noise, masks, and selection rules
are specified in appendix~\ref{app:exp3_details}.

\begin{table}[htbp]
\centering
\small
\setlength{\tabcolsep}{3pt}
\caption{Experiment 3 recovery (three seeds), with partial-in-time
density observations in the $50\%$ rows. JFB-25 is included only in the
noisy, partial-in-time setting. Best-by-loss errors agree with the final errors
within $0.01$ percentage point and are omitted.}
\label{tab:exp3}
\begin{tabular}{p{0.24\linewidth}lcrr}
\toprule
{Setting} & {Method} & {Final Err.\ (\%)} & {Time (s)} & {Mem.\ (MB)} \\
\midrule
\multirow{2}{*}{Clean, $100\%$ obs}   & AD       & $0.63 \pm 0.04$ & $94$ & $2570$ \\
                                       & JFB-10  & $0.78 \pm 0.06$ & $\mathbf{37}$ & $\mathbf{280}$ \\
\midrule
\multirow{2}{*}{$\gamma_n{=}0.01$, $100\%$ obs}  & AD       & $1.36 \pm 0.10$ & $92$ & $2570$ \\
                                                  & JFB-10  & $1.66 \pm 0.23$ & $\mathbf{35}$ & $\mathbf{280}$ \\
\midrule
\multirow{2}{*}{Clean, $50\%$ obs}    & AD       & $0.64 \pm 0.04$ & $94$ & $2570$ \\
                                       & JFB-10  & $0.78 \pm 0.13$ & $\mathbf{36}$ & $\mathbf{280}$ \\
\midrule
\multirow{3}{*}{$\gamma_n{=}0.05$, $50\%$ obs}   & AD       & $9.50 \pm 2.84$ & $94$ & $2570$ \\
                                                  & JFB-10  & $10.82 \pm 2.32$ & $\mathbf{34}$ & $\mathbf{280}$ \\
                                                  & JFB-25  & $9.49 \pm 2.91$  & $46$ & $662$ \\
\bottomrule
\end{tabular}
\end{table}

Figure~\ref{fig:exp3_tracked_depth} shows reconstructions from a
representative seed. JFB-10 is within $0.3$ percentage points of AD in
the easier settings; under heavy noise and partial-in-time observation,
JFB-25 is indistinguishable from AD within the three-seed spread while
using about one quarter of AD's peak memory and about half its runtime.

\begin{figure}[htbp!]
\centering
\includegraphics[width=0.9\linewidth]{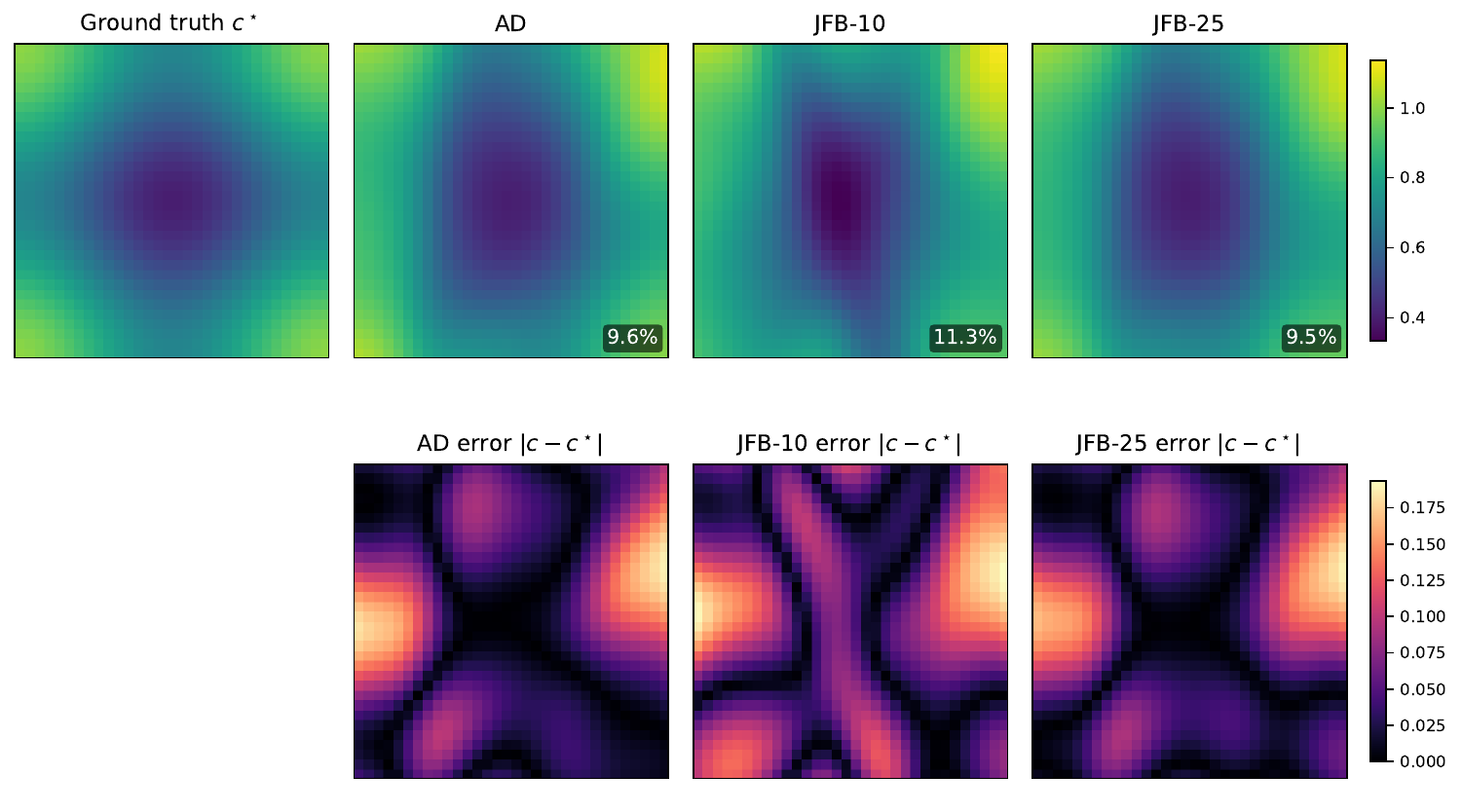}
\caption{Experiment 3 noisy, partial-in-time setting for a
representative seed.
Top: ground truth and AD, JFB-10, and JFB-25 reconstructions.
Bottom: pointwise errors on a shared scale.}
\label{fig:exp3_tracked_depth}
\end{figure}

The depth dependence in Table~\ref{tab:exp3} is consistent with the 2D
alignment probe reported in appendix~\ref{app:exp3_alignment}, which is
based on a single representative run.

\refstepcounter{experiment}\label{sec:exp4}
\paragraph{Shared-cost recovery across multiple instances
(Experiment~\theexperiment)}

Ten MFG instances share a two-bump inverse-cost map but have distinct
Gaussian-mixture endpoint densities. Observations are generated on a $N_x^2=128^2$,
$N_t=64$ grid, restricted to the $N_x^2=64^2$, $N_t=32$ reconstruction grid,
and sampled at $25\%$ of output times with noise level $\gamma_n=0.2$, so recovery is tested under discretization mismatch.
AD, JFB-10, and JFB-25 train a width-64 MLP for $2000$ mini-batch steps
with batch size two and $N_{\mathrm{itr}}=100$; full details are in
appendix~\ref{app:exp4_details}.

\begin{table}[htbp]
\centering
\caption{Experiment 4 mini-batch recovery (three seeds). Errors are
final-iterate mean $\pm$ sample standard deviation; time and memory are seed means.}
\label{tab:exp4}
\begin{tabular}{lccr}
\toprule
{Method} & {Final Err.\ (\%)} & {Time (min)} & {Mem.\ (GB)} \\
\midrule
AD     & $8.54 \pm 1.83$ & $21.8$          & $20.2$ \\
\midrule
JFB-10 & $8.55 \pm 3.40$ & $\mathbf{6.2}$  & $\mathbf{2.3}$ \\
\midrule
JFB-25 & $\mathbf{6.11 \pm 1.88}$ & $8.9$   & $5.2$ \\
\bottomrule
\end{tabular}
\end{table}

Table~\ref{tab:exp4} reports final-iterate results. JFB-25 has the
lowest mean recovery error, $6.11\%$; AD and JFB-10 give $8.54\%$ and
$8.55\%$, respectively. JFB-10 runs $3.5\times$ faster than AD and
reduces peak memory by a factor of $8.9$, while JFB-25 has intermediate
runtime and memory. Because the outer updates are minibatched, JFB-$r$
is here a biased stochastic-gradient method.
Figure~\ref{fig:exp4_recovery} shows the final reconstructions.

\begin{figure}[htbp!]
\centering
\includegraphics[width=0.9\linewidth]
{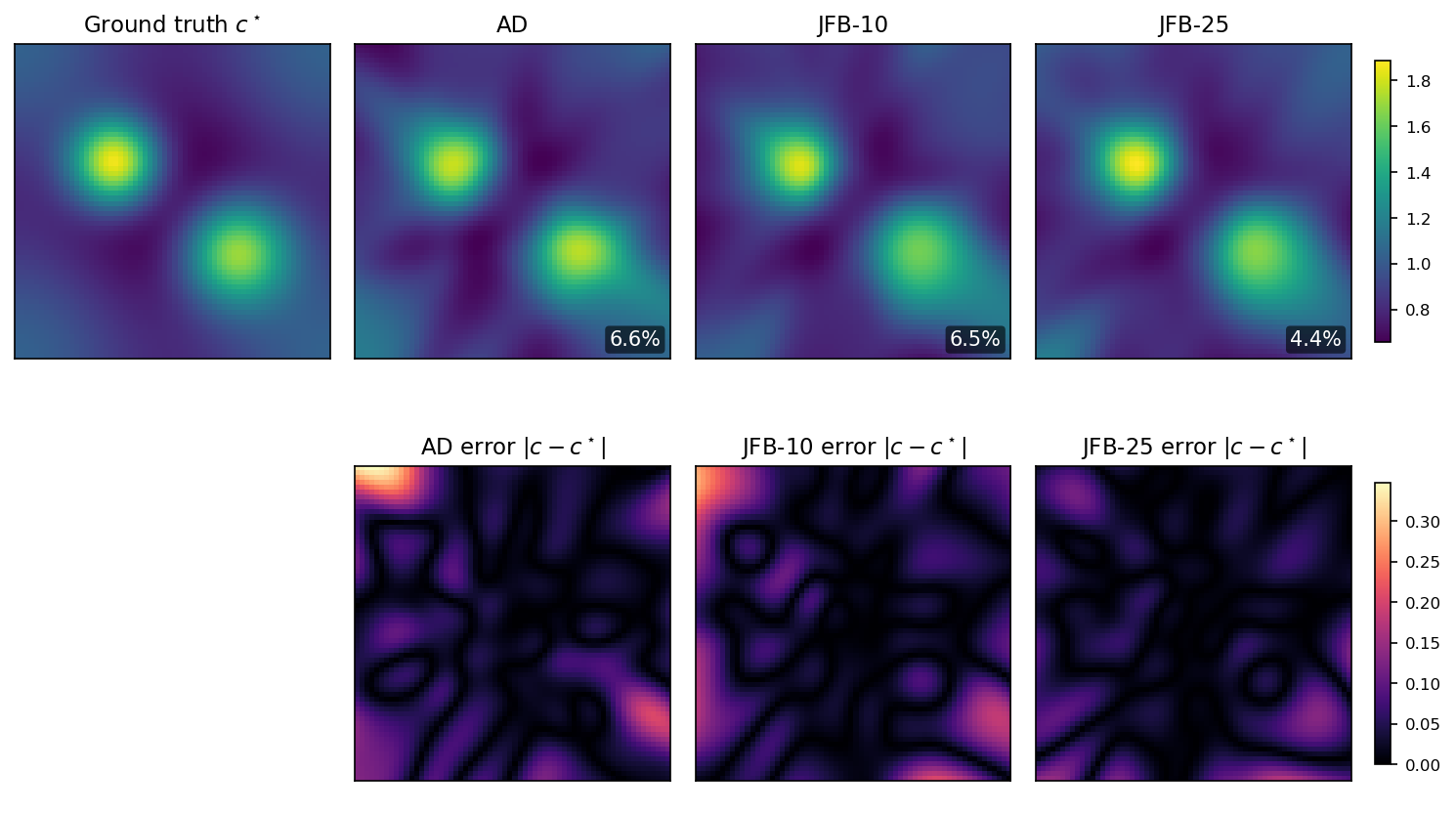}
\caption{Experiment 4 ground truth and final-iterate reconstructions
for a representative seed ($123$) under fine-to-coarse observations,
with pointwise errors on a shared scale. Panel labels give this seed's
final errors; three-seed means are in Table~\ref{tab:exp4}.}
\label{fig:exp4_recovery}
\end{figure}

\section{Conclusion}
\label{sec:conclusion}

We introduced a JFB-PDHG method for inverse mean-field games.
It avoids equilibrium-adjoint linear solves and full-trajectory reverse-mode storage; recorded tape storage scales with the tracked depth $r$.

The analysis identifies the exact-proximal PDHG map $T_\theta$ as an exact metric resolvent, proves directional stability without flux strict complementarity, and interprets JFB-$r$, under a locally fixed active set at an exact equilibrium detach point, as a truncated adjoint iteration with large-depth consistency.

In the controlled 1D test, the noisy test with partial-in-time observations, and the multi-instance test, selected moderate depths reduced runtime and peak memory relative to full unrolling while retaining recovery errors close to those of AD.
At $128^2$, JFB-10 was about $5.2\times$ faster than checkpointed AD at comparable peak memory, with higher recovery error.
The fixed-active-set exact-detach theory gives consistency with depth.
Future work includes approximate-detach bounds, adaptive depth, generalized kink sensitivity, and extending the framework to higher-dimensional inverse problems arising in control and dynamics~\cite{koltun2012continuous, daosud2005neural}, as well as inverse transport and flow reconstruction~\cite{stuart2020inverse, onken2021ot, gonzalez2024nonlinear, andrade2024sparsistency}.

\section*{Acknowledgments}
ChatGPT (OpenAI) was used to polish the authors' written text for spelling, grammar, and style.
Claude (Anthropic) was used to assist with JAX-based acceleration of portions of the numerical implementation, specifically translating and vectorizing reference code for the PDHG and JFB updates.
All AI-assisted text and code were reviewed and tested by the authors, who take full responsibility for the final manuscript.

\appendix

\section{Block decomposition of the PDHG iteration}
\label{app:block_pdhg}

This appendix records the explicit operator blocks used by the implemented PDHG iteration.
The main text uses their compact form based on $T_x$, $S$, and $K$.

\subsection{Active primal/dual blocks and dimensions}
\label{app:block_dims}
The discrete grid is cell-centered in space with $N_x$ cells of width $h_x$ and vertex-in-time with $N_t$ intervals of width $h_t$.
The density $\rho^k_i$ lives on slices $k=0,\ldots,N_t$, the upwind flux components $m^{+,j}_i\ge 0$ and $m^{-,j}_i\le 0$ live on intervals $j=0,\ldots,N_t-1$, and the dual $\phi^j_i$ lives on the same interval grid as the continuity residual.

The initial slice is fixed, $\rho^0\equiv\mu_0$, and is not a primal degree of freedom.
The two boundary slots $m^{+,j}_{N_x-1}$ and $m^{-,j}_0$ are masked to zero by the no-flux convention.
The active primal vector is
\[
x=(\rho^{1:N_t},\,m^+_{\mathrm{act}},\,m^-_{\mathrm{act}})\in\mathbb{R}^n,
\qquad
n=N_xN_t+2(N_x-1)N_t,
\]
and the dual vector is
\[
\phi\in\mathbb{R}^{N_c},\qquad N_c=N_xN_t.
\]
In $d$ space dimensions each axis contributes its own upwind pair $m^{\ell,\pm}$ on interior faces, so each cell carries $q=2d$ split flux components.
The closed-form expressions in appendices~\ref{app:K_structure} and~\ref{app:Tx_S} are written in $d=1$ for notational economy and extend componentwise.

\subsection{Single-row \texorpdfstring{$K$}{K} and its adjoint}
\label{app:K_structure}
The constraint operator $K\in\mathbb{R}^{N_c\times n}$ implements the discrete continuity equation with the IC absorbed into the right-hand side.
With the row-block decomposition of the active primal,
\begin{equation}
K=\bigl[\;D_t\;\;D_{m^+}\;\;D_{m^-}\;\bigr],
\label{eq:K_blocks}
\end{equation}
the action is
\[
(Kx)^j_i=\frac{\rho^{j+1}_i-\rho^j_i}{h_t}+(\nabla\cdot m)^j_i,
\qquad j=0,\ldots,N_t-1,
\]
where the upwind divergence is
\[
(\nabla\cdot m)^j_i
=\frac{m^{+,j}_i-m^{+,j}_{i-1}}{h_x}
+\frac{m^{-,j}_{i+1}-m^{-,j}_i}{h_x}
\]
with Neumann (no-flux) boundary conditions $m^{+,j}_{-1}=m^{+,j}_{N_x-1}=0$ and $m^{-,j}_{N_x}=m^{-,j}_0=0$.
The right-hand side of the constraint is $b\in\mathbb{R}^{N_xN_t}$ with $b^0_i=\mu_{0,i}/h_t$ and $b^j_i=0$ for $j\ge 1$, so $Kx=b$.

The temporal block $D_t$ acts on the active density slices by $(D_t\rho)^j_i=(\rho^{j+1}_i-\rho^j_i)/h_t$ with $\rho^0:=0$ inside the operator (the $\mu_{0,i}/h_t$ shift is moved to $b$).
On the $N_t$ active slices $D_t$ is square and lower-bidiagonal with nonzero diagonal entries $1/h_t$, hence invertible, and has full row rank.
Thus $K$ has full row rank.

The adjoint $K^\top\phi$ maps the interval-defined dual back to the primal blocks with no slice-to-interval lift (since $Kx$ and $\phi$ share the same $(N_x,N_t)$ shape); the flux rows are the corresponding upwind differences of $\phi$ in space.
The one property used later is that the \emph{density} rows involve the temporal block alone,
\[
(K^\top\phi)\big|_{\rho^k_i}
=\frac{\phi^{k-1}_i-\phi^k_i}{h_t}
\;\; (k<N_t),
\qquad
(K^\top\phi)\big|_{\rho^{N_t}_i}
=\frac{\phi^{N_t-1}_i}{h_t},
\]
that is, $(K^\top\phi)|_\rho=D_t^\top\phi$ sitewise, with no contribution from the flux columns.

\subsection{Step-size operators and the exact Schur preconditioner}
\label{app:Tx_S}
The primal and dual step operators are scalar multiples of the identity and the exact Schur preconditioner respectively,
\begin{equation}
T_x=\tau I_n,
\qquad
S=\sigma W^{-1},
\qquad
W:=KK^\top.
\label{eq:S_explicit}
\end{equation}
Because $K$ has a single continuity row block, $W$ is the block sum of three contributions
\[
W=D_tD_t^\top+D_{m^+}D_{m^+}^\top+D_{m^-}D_{m^-}^\top,
\]
and the spatial-temporal separability of the divergence operator gives the tensor-product form
\[
W=I_x\otimes B_t+2L_x\otimes I_t,
\qquad
B_t:=D_tD_t^\top,
\]
where $L_x\in\mathbb{R}^{N_x\times N_x}$ is the cell-centered 1D Neumann graph Laplacian and $B_t\in\mathbb{R}^{N_t\times N_t}$ is the interval-time tridiagonal
\[
B_t=\frac{1}{h_t^2}
\begin{pmatrix}
1 & -1 \\
-1 & 2 & -1 \\
& -1 & 2 & \ddots \\
&& \ddots & \ddots & -1 \\
&&& -1 & 2
\end{pmatrix}.
\]
The first diagonal entry is $1$ rather than $2$ because $\rho^0$ is fixed (only one off-diagonal coupling at slice $\rho^1$), and this single change makes $D_t$ full row rank and hence $W\succ 0$ on the entire dual space.
The eigenvalues of $W$ are
\[
\nu_{i,\ell}(W)=\nu^t_\ell+2\nu^x_i>0,
\qquad
\nu^t_\ell=\frac{2}{h_t^2}\Bigl(1-\cos\tfrac{(2\ell+1)\pi}{2N_t+1}\Bigr),
\qquad
\nu^x_i=\frac{2}{h_x^2}\Bigl(1-\cos\tfrac{\pi i}{N_x}\Bigr).
\]
We apply $W^{-1}$ to a dual residual using a DCT in space and an eigendecomposition of the $N_t\times N_t$ time block $B_t$.

\subsubsection{Closed-form step-size check}
\label{app:stepsize_check}
With $T_x=\tau I$ and $S=\sigma W^{-1}$, $W=KK^\top$, the step-size quantity of Assumption~\ref{asm:pdhg_stepsizes} is $\|S^{1/2}KT_x^{1/2}\|_2^2=\tau\sigma\|W^{-1/2}K\|_2^2$, and full row rank of $K$ gives $\|W^{-1/2}K\|_2^2=1$ exactly.
The assumption therefore reduces to $\tau\sigma<1$, satisfied by the reconstruction choice $\tau=\sigma=0.99$ with no numerical norm estimation.

\subsection{Per-cell proximal step}
\label{app:prox}
After the linear PDHG step the proximal map decouples across $(i,j)$ because the kinetic energy at interval $j$ depends only on $\rho^{j+1}$ (backward-Euler pairing).
For each $(i,j)$ the prox jointly updates the cell-time triple $(\rho^{j+1}_i,m^{+,j}_i,m^{-,j}_i)$ by minimizing
\begin{align*}
\frac{(m^+)^2+(m^-)^2}{2c\rho}
+\gamma_I\rho\log\rho
+\mathbf{1}_{\{j+1=N_t\}}\frac{\gamma_T}{h_t}\rho(\log\rho-\log\mu_{1,i})
+\\ \frac{(\rho-\tilde\rho)^2}{2\tau}
+\frac{(m^+-\tilde m^+)^2}{2\tau}
+\frac{(m^--\tilde m^-)^2}{2\tau}
\end{align*}
over $\rho\ge 0$, $m^+\ge 0$, $m^-\le 0$, where $(\tilde\rho,\tilde m^\pm)$ are the gradient-step inputs and $c=c_\theta(s_i)$.
This is exactly $\prox_{\tau f_\theta}$ for the solver-normalized objective in~\eqref{eq:disc_objective}; the implementation, dual variable, KKT operator, and all finite-iteration sensitivity objects use this same normalization.
The flux components are recovered in closed form from $\rho$,
\[
m^{+,*}=\max(\tilde m^+,0)\,\frac{c\rho}{c\rho+\tau},\qquad
m^{-,*}=\min(\tilde m^-,0)\,\frac{c\rho}{c\rho+\tau},
\]
and the reduced 1D objective in $\rho$,
\begin{align*}
\Phi(\rho)=\frac{\kappa}{2(c\rho+\tau)}+\gamma_I\rho\log\rho
+\mathbf{1}_{\{j+1=N_t\}}\tfrac{\gamma_T}{h_t}(\rho\log\rho-\rho\log\mu_{1,i})
+\frac{(\rho-\tilde\rho)^2}{2\tau},
\\
\kappa:=[\max(\tilde m^+,0)]^2+[\min(\tilde m^-,0)]^2,
\end{align*}
is approximated in the implementation by a fixed budget of $12$ damped Newton iterations, with the step safeguarded so that $\rho$ never crosses the floor $\rho\ge 10^{-12}$.

\section{Proof of Lemma~\ref{lemma:f_theta_scvx}}
\label{app:scvx_proof}

The smooth objective is separable across space-time cells.
From~\eqref{eq:disc_objective} and the per-cell form of the terminal functional (appendix~\ref{app:prox}),
\[
\cE_\theta(x)
=
\sum_{j=0}^{N_t-1}\sum_i
e_{ij}\bigl(\rho^{j+1}_i,\,m^{+,j}_i,\,m^{-,j}_i\bigr),
\]
\[
e_{ij}(\rho,m^+,m^-)
:=
\frac{(m^+)^2+(m^-)^2}{2\,c_i\,\rho}
+\gamma_I\,\rho\log\rho
+\mathbf{1}_{\{j+1=N_t\}}\,\frac{\gamma_T}{h_t}\,
\rho\bigl(\log\rho-\log\mu_{1,i}\bigr),
\]
with $c_i=c_\theta(s_i)$; masked boundary flux slots simply drop the corresponding variables.
Distinct cells share no variables (the backward-Euler pairing ties $m^{\pm,j}$ to $\rho^{j+1}$ only), so $\nabla^2\cE_\theta(x)$ is block diagonal with one $3\times3$ (or smaller) block $\nabla^2e_{ij}$ per cell, and it suffices to bound each block uniformly over $x\in U$.

Fix a cell and write $e=e_{ij}$, $c=c_i$,
\[
\alpha:=\frac{1}{c\rho},
\qquad
w:=\Bigl(\frac{m^+}{\rho},\frac{m^-}{\rho}\Bigr)\in\mathbb{R}^2,
\qquad
\beta:=\gamma_I+\mathbf{1}_{\{j+1=N_t\}}\frac{\gamma_T}{h_t}.
\]
A direct computation gives, for every direction $(a,b)\in\mathbb{R}\times\mathbb{R}^2$,
\begin{equation}
\bigl\langle(a,b),\,\nabla^2e(\rho,m^+,m^-)\,(a,b)\bigr\rangle
=
\alpha\,\bigl|b-a\,w\bigr|^2+\frac{\beta}{\rho}\,a^2 .
\label{eq:percell_form}
\end{equation}
Indeed, the kinetic perspective term contributes $\partial^2_{mm}=\alpha I_2$, $\partial^2_{m\rho}=-\alpha w$, $\partial^2_{\rho\rho}=\alpha|w|^2$, whose quadratic form is $\alpha|b-aw|^2$, while the entropy and terminal terms contribute $\beta/\rho$ in the $\rho\rho$ entry.
Both terms in~\eqref{eq:percell_form} are nonnegative, and the form vanishes only if $a=0$ (since $\beta\ge\gamma_I>0$) and then $b=0$: the kinetic degeneracy along the scaling ray $(\rho,m)$ is exactly what the entropy curvature removes.

\emph{Strong convexity.} Decompose $b$ into components parallel and orthogonal to $w$.
The orthogonal component sees the eigenvalue $\alpha$, while on the $(a,\,b_{\parallel})$ plane the form~\eqref{eq:percell_form} has matrix
\[
Q=
\begin{pmatrix}
\alpha|w|^2+\beta/\rho & -\alpha|w|\\
-\alpha|w| & \alpha
\end{pmatrix},
\qquad
\det Q=\frac{\alpha\beta}{\rho},
\qquad
\operatorname{tr}Q=\alpha\bigl(1+|w|^2\bigr)+\frac{\beta}{\rho}.
\]
For a $2\times2$ symmetric positive-semidefinite matrix, $\lambda_{\min}(Q)\ge\det Q/\operatorname{tr}Q$, hence
\[
\lambda_{\min}\bigl(\nabla^2e\bigr)
\;\ge\;
\min\Bigl\{\alpha,\;
\frac{\alpha\beta/\rho}{\alpha(1+|w|^2)+\beta/\rho}\Bigr\}
=
\frac{\alpha\beta/\rho}{\alpha(1+|w|^2)+\beta/\rho}.
\]
The right-hand side is increasing in $\alpha$ and in $\beta/\rho$ and decreasing in $|w|^2$.
On $U$, Assumption~\ref{asm:forward_regularity} gives the uniform bounds
\[
\alpha\ge\alpha_0:=\frac{1}{c_{\max}\rho_{\max}},
\qquad
\frac{\beta}{\rho}\ge\beta_0:=\frac{\gamma_I}{\rho_{\max}},
\qquad
|w|^2\le w_0^2:=\frac{2\,m_{\text{max}}^2}{\rho_{\min}^2},
\]
so that, uniformly over cells, $x\in U$, and $\theta$ (only the bounds $c_{\min}\le c_\theta\le c_{\max}$ enter),
\[
\nabla^2\cE_\theta(x)
\;\succeq\;
\mu\,I,
\qquad
\mu:=
\frac{\alpha_0\beta_0}{\alpha_0\bigl(1+w_0^2\bigr)+\beta_0}
\;>\;0 .
\]
Since $U$ is convex, $\cE_\theta$ is $\mu$-strongly convex on $U$.

\emph{Smoothness.} In the other direction, $\lambda_{\max}(\nabla^2e)\le\operatorname{tr}(\nabla^2e) =\alpha(2+|w|^2)+\beta/\rho$, so with $\alpha_1:=1/(c_{\min}\rho_{\min})$ and $\beta_1:=(\gamma_I+\gamma_T/h_t)/\rho_{\min}$,
\[
\nabla^2\cE_\theta(x)\;\preceq\;L_f\,I,
\qquad
L_f:=\alpha_1(2+w_0^2)+\beta_1,
\]
and $\nabla\cE_\theta$ is $L_f$-Lipschitz on $U$.

The preceding calculation is written in the 1D notation of~\eqref{eq:disc_objective}.
In $d$ space dimensions each cell carries $q=2d$ split flux components (appendix~\ref{app:block_dims}), the identity~\eqref{eq:percell_form} holds unchanged with $b,w\in\mathbb{R}^{q}$ and $w=m/\rho$, and only the plane spanned by the density direction and $w$ is coupled; the remaining $q-1$ flux directions are eigendirections with eigenvalue $\alpha$.
The bounds then hold with $w_0^2=2d\,m_{\text{max}}^2/\rho_{\min}^2$ and with the trace bound $\alpha_1(q+w_0^2)+\beta_1$ in $L_f$; for $d=1$ these reduce to the one-dimensional constants.

Finally, under the locally fixed active set of Assumption~\ref{asm:local_chart} the matrices are principal submatrices of the full matrices (active flux columns removed), and eigenvalue interlacing preserves both bounds, so the same $\mu$ and $L_f$ serve on the reduced coordinates.
\qed

\section{Proof of Theorem~\ref{thm:pdhg_resolvent}}
\label{app:pdhg_resolvent_proof}

\begin{proof}[Proof of Theorem~\ref{thm:pdhg_resolvent}]
Uniqueness and single-valuedness of the solution $z^+$ of the inclusion are part of Corollary~\ref{cor:resolvent_facts} (appendix~\ref{app:resolvent_facts}); here we verify that the PDHG update solves it.
Write $z=(x,\phi)$, $z^+=(x^+,\phi^+)$.
With $\cA_\theta$ from~\eqref{eq:kkt_operator} and $\cM$ from~\eqref{eq:M_metric}, the inclusion $\cM(z^+-z)+\cA_\theta(z^+)\ni0$ reads, in blocks,
\begin{align}
0&\in
T_x^{-1}(x^+-x)+K^\top(\phi^+-\phi)
+\partial f_\theta(x^+)+K^\top\phi^+,
\label{eq:res_block1}\\
0&=
K(x^+-x)+S^{-1}(\phi^+-\phi)+b-Kx^+.
\label{eq:res_block2}
\end{align}
In the dual block~\eqref{eq:res_block2} the $Kx^+$ terms cancel, leaving $S^{-1}(\phi^+-\phi)=Kx-b$, i.e.,
\[
\phi^+=\phi+S(Kx-b),
\]
which is exactly the dual update~\eqref{eq:pdhg_dual}.
In the primal block~\eqref{eq:res_block1}, combining the two dual terms gives
\[
0\in
T_x^{-1}(x^+-x)+\partial f_\theta(x^+)
+K^\top(2\phi^+-\phi).
\]
This is precisely the first-order optimality condition of
\[
x^+=\argmin_{x'}\Bigl\{f_\theta(x')
+\tfrac12\bigl\|x'-\bigl(x-T_xK^\top(2\phi^+-\phi)\bigr)\bigr\|_{T_x^{-1}}^2\Bigr\}
=\prox_{T_x f_\theta}\!\bigl(x-T_xK^\top(2\phi^+-\phi)\bigr),
\]
which is exactly the extrapolated primal update~\eqref{eq:pdhg_primal}.
Thus $z^+=T_\theta(z)$ solves the inclusion, and conversely every solution of the inclusion satisfies the two PDHG update formulas, so $T_\theta=R_{\theta,1}$.
\end{proof}

\subsection{Metric positivity and resolvent facts}
\label{app:resolvent_facts}

We first record that $\cM\succ0$ under Assumption~\ref{asm:pdhg_stepsizes}, as used throughout.
Let
\[
C :=
\begin{pmatrix}
I & 0 \\
SK & I
\end{pmatrix},
\qquad
\widetilde{\cM}:=
\begin{pmatrix}
Q & 0 \\
0 & S^{-1}
\end{pmatrix},
\qquad
Q := T_x^{-1}-K^\top SK .
\]
A direct computation gives the congruence
\begin{equation}
C^\top \widetilde{\cM}\, C
=
\begin{pmatrix}
Q + K^\top SK & K^\top \\
K & S^{-1}
\end{pmatrix}
= \cM .
\label{eq:metric_pullback}
\end{equation}
Now $S=\sigma W^{-1}\succ0$ because $W=KK^\top\succ0$ (appendix~\ref{app:Tx_S}), and Assumption~\ref{asm:pdhg_stepsizes} says $\|S^{1/2}KT_x^{1/2}\|_2^2<1$, i.e., $I-T_x^{1/2}K^\top SKT_x^{1/2}\succ0$; conjugating by $T_x^{-1/2}$ gives $Q\succ0$.
Hence $\widetilde{\cM}\succ0$ and, since $C$ is invertible, $\cM\succ0$ by~\eqref{eq:metric_pullback}.

\paragraph{Proof of Corollary~\ref{cor:resolvent_facts}}
Because $\cM\succ0$, the pairing $\langle u,v\rangle_{\cM}:=\langle\cM u,v\rangle$ defines an equivalent Hilbert-space structure.
Maximal monotonicity of $\cA_\theta$ in the Euclidean pairing is equivalent to maximal monotonicity of $\cM^{-1}\cA_\theta$ in this metric.
Minty's theorem~\cite{minty1962monotone,bauschke2017convex} therefore implies that $R_{\theta,\lambda}=(I+\lambda\cM^{-1}\cA_\theta)^{-1}$ is everywhere defined and single-valued, proving (1).

For (2), let $p=R_{\theta,\lambda}z$ and $q=R_{\theta,\lambda}w$.
By definition,
\[
\frac1\lambda \cM(z-p)\in\cA_\theta(p),
\qquad
\frac1\lambda \cM(w-q)\in\cA_\theta(q).
\]
Monotonicity of $\cA_\theta$ applied to this pair gives
\[
\bigl\langle \cM[(z-w)-(p-q)],\;p-q\bigr\rangle\ge0,
\]
that is, $\langle z-w,\,p-q\rangle_{\cM}\ge\|p-q\|_{\cM}^2$, which is firm nonexpansiveness in the $\cM$-inner product.

For (3), $R_{\theta,\lambda}(z^\star)=z^\star$ if and only if $z^\star\in z^\star+\lambda\cM^{-1}\cA_\theta(z^\star)$, i.e., $0\in\cA_\theta(z^\star)$.
\hfill$\square$

\section{Proof of Theorem~\ref{thm:nonsmooth_kernel}}
\label{app:nonsmooth_kernel}

Throughout, $(\theta_0,z_0)$ is fixed, $x_0$ and $v_0=-\nabla\cE_{\theta_0}(x_0)-K^\top\phi_0\in N_C(x_0)$ are as in section~\ref{subsec:nonsmooth}, and slots $i$ of the primal space are classified at $(x_0,v_0)$ as \emph{inactive} ($x_{0,i}$ interior to its factor, $v_{0,i}=0$), \emph{strongly active} ($x_{0,i}$ on the boundary, $v_{0,i}\neq0$), or \emph{degenerate} ($x_{0,i}$ on the boundary, $v_{0,i}=0$).
The critical cone $\mathcal{K}=\mathcal{T}_C(x_0)\cap v_0^\perp$ is correspondingly $\mathbb{R}$, $\{0\}$, or a closed half-line in each coordinate, and its polar satisfies, coordinatewise,
\begin{equation}
(\mathcal{K}^\circ)_i=
\begin{cases}
\{0\} & i \text{ inactive},\\
\mathbb{R} & i \text{ strongly active},\\
\text{a half-line} & i \text{ degenerate}.
\end{cases}
\label{eq:polar_slots}
\end{equation}

\paragraph{Step 1: local reduction}
Because $C$ is a product of copies of $\mathbb{R}$, $[0,\infty)$, and $(-\infty,0]$, the normal-cone graph satisfies the \emph{exact} local reduction identity
\begin{equation}
\mathrm{gph}\,N_C\;\cap\;\mathcal{U}
=\bigl[(x_0,v_0)+\mathrm{gph}\,N_{\mathcal{K}}\bigr]\;\cap\;\mathcal{U}
\label{eq:reduction}
\end{equation}
for some neighborhood $\mathcal{U}$ of $(x_0,v_0)$.
This is the polyhedral reduction lemma; see \cite[Lemma~2E.4]{dontchev2014implicit} and \cite[Ch.~13]{rockafellarwets1998}.
For our orthant product, check each coordinate separately.
At an unconstrained or inactive slot both sides read $\mathbb{R}\times\{0\}$ locally; at a strongly active slot both read $\{0\}\times\mathbb{R}$ locally; at a degenerate slot the graph of $N_{[0,\infty)}$ is a cone with vertex at the origin, so it coincides with its own translate by $(0,0)$ and with $\mathrm{gph}\,N_{\mathcal{K}_i}$ for $\mathcal{K}_i=[0,\infty)$ (mirrored for $m^-$ slots).
Products of the per-coordinate neighborhoods give~\eqref{eq:reduction}.
Since $\cE_{\theta_0}$ is $C^2$ near $x_0$ (Lemma~\ref{lemma:f_theta_scvx}), the graphical derivative of $\cA_{\theta_0}$ at $(z_0,0)$ is the Jacobian of the smooth part plus the derivative~\eqref{eq:reduction} of the set-valued part, which is exactly~\eqref{eq:DA_def}.

\paragraph{Step 2: maximal monotonicity of $D\cA$}
Write $D\cA=G_0+\partial\Phi$ with
\[
G_0 h=\begin{pmatrix}\nabla^2\cE_{\theta_0}(x_0)h_x+K^\top h_\phi\\
-Kh_x\end{pmatrix},
\qquad
\Phi(h)=I_{\mathcal{K}\times\mathbb{R}^{N_c}}(h),
\]
so that $\partial\Phi(h)=N_{\mathcal{K}}(h_x)\times\{0\}$.
$G_0$ is linear, everywhere defined, and monotone (the skew coupling cancels in $\langle G_0h,h\rangle=h_x^\top\nabla^2\cE_{\theta_0}h_x\ge0$), hence maximal monotone; $\partial\Phi$ is the subdifferential of a proper closed convex function, hence maximal monotone; and the sum of two maximal monotone operators one of which has full domain is maximal monotone \cite[Cor.~25.5]{bauschke2017convex}.

\paragraph{Step 3: properties of $\Rz$}
By Minty's theorem in the Hilbert space $(\mathbb{R}^{n+N_c},\langle\cdot,\cdot\rangle_{\cM})$ applied to the maximal monotone operator $\cM^{-1}D\cA$, the resolvent $\Rz=(I+\cM^{-1}D\cA)^{-1}$ is everywhere defined and single-valued, and the standard computation (as in appendix~\ref{app:resolvent_facts}, with $D\cA$ in place of $\cA_\theta$) gives firm nonexpansiveness in $\langle\cdot,\cdot\rangle_{\cM}$.
The graph of $D\cA$ is a cone: for $t>0$, $N_{\mathcal{K}}(th_x)=N_{\mathcal{K}}(h_x)$ and $tN_{\mathcal{K}}(h_x)=N_{\mathcal{K}}(h_x)$ because normal cones to cones are cones, so $D\cA(th)=tD\cA(h)$; hence $\Rz(th)=t\Rz(h)$ for $t\ge0$ and $\Rz(0)=0$.
The fixed-point identity $\Rz(h)=h\iff0\in D\cA(h)$ is immediate from the defining inclusion, and firm nonexpansiveness with $k=0$ gives $\|\Rz(h)\|_{\cM}\le\|h\|_{\cM}$, hence $\|\Rz^n(h)\|_{\cM}\le\|h\|_{\cM}$ for all $n$.

For semi-differentiability, let $u$ be small, $s:=T_{\theta_0}(z_0+u)-z_0$ (so $\|s\|_{\cM}\le\|u\|_{\cM}$ by nonexpansiveness and $T_{\theta_0}(z_0)=z_0$), and $\tilde s:=\Rz(u)$.
Substituting $z^+=z_0+s$ into the resolvent inclusion $\cM(z^+-z_0-u)+\cA_{\theta_0}(z^+)\ni0$, using the equilibrium identities $\nabla\cE_{\theta_0}(x_0)+K^\top\phi_0=-v_0$ and $Kx_0=b$, the Taylor expansion $\nabla\cE_{\theta_0}(x_0+s_x)=\nabla\cE_{\theta_0}(x_0) +\nabla^2\cE_{\theta_0}(x_0)s_x+\rho(s_x)$ with $\|\rho(s_x)\|=o(\|s_x\|)$, and the exact identity~\eqref{eq:reduction}, the inclusion becomes
\[
\cM(s-u)+D\cA(s)+\bigl(\rho(s_x),0\bigr)\ni0 .
\]
(The use of~\eqref{eq:reduction} here is legitimate for all sufficiently small $u$: the normal-cone element selected by the perturbed inclusion, $w=-T_x^{-1}(s_x-u_x)-K^\top(s_\phi-u_\phi) -\nabla\cE_{\theta_0}(x_0+s_x)-K^\top(\phi_0+s_\phi) \in N_C(x_0+s_x)$, is continuous in $(u,s)$ with value $v_0$ at $u=0$, and $\|s\|_{\cM}\le\|u\|_{\cM}$, so the pair $(x_0+s_x,w)$ lies in the neighborhood $\mathcal U$ of~\eqref{eq:reduction} once $u$ is small.)
Compare this inclusion with
\[
\cM(\tilde s-u)+D\cA(\tilde s)\ni0.
\]
Monotonicity of $D\cA$ gives
\[
\|s-\tilde s\|_{\cM}
\le\|\cM^{-1}(\rho(s_x),0)\|_{\cM}
=o(\|s\|)=o(\|u\|).
\]
The estimate is uniform over directions because the Taylor remainder of the $C^2$ function $\cE_{\theta_0}$ is uniform on compact sets.
This proves part (1); see \cite{rockafellar1989proto} for the general proto-differentiability framework behind this argument.

\paragraph{Step 4: the kernel is trivial}
Let $0\in D\cA(h)$: there is $w\in N_{\mathcal{K}}(h_x)$ with
\[
\nabla^2\cE_{\theta_0}(x_0)h_x+K^\top h_\phi+w=0,
\qquad
Kh_x=0 .
\]
Pair the first equation with $h_x$.
The coupling term vanishes, $\langle K^\top h_\phi,h_x\rangle=\langle h_\phi,Kh_x\rangle=0$, and so does $\langle w,h_x\rangle$: for a closed convex cone $\mathcal{K}$, $w\in N_{\mathcal{K}}(h_x)$ means $h_x\in\mathcal{K}$, $w\in\mathcal{K}^\circ$, and $\langle w,h_x\rangle=0$.
Hence $h_x^\top\nabla^2\cE_{\theta_0}(x_0)h_x=0$, and strong convexity (Lemma~\ref{lemma:f_theta_scvx}) forces $h_x=0$.

With $h_x=0$, the first equation reads $-K^\top h_\phi=w\in N_{\mathcal{K}}(0)=\mathcal{K}^\circ$.
Consider only the \emph{density} slots of this vector equation.
Density slots are inactive because $\rho\ge\rho_{\min}>0$ on $U$ (Assumption~\ref{asm:forward_regularity}(1)).
By~\eqref{eq:polar_slots} the polar vanishes there, and by the block structure of $K$ (appendix~\ref{app:K_structure}) the density rows of $K^\top h_\phi$ are exactly $D_t^\top h_\phi$, sitewise, with no contribution from the flux columns.
Hence $D_t^\top h_\phi=0$ at every site.
$D_t$ is square and lower-bidiagonal with diagonal entries $1/h_t$ (appendix~\ref{app:K_structure}), so $D_t^\top$ is injective by back-substitution starting from the terminal row, and $h_\phi=0$.
No information about the flux rows was used: the argument is indifferent to the flux active and degenerate sets.
This proves part (2).

\paragraph{Step 5: geometric decay after finite grouping}
$\Rz$ is firmly nonexpansive on the $\cM$-Hilbert space, hence $\tfrac12$-averaged, with $\operatorname{Fix}\Rz=\ker D\cA=\{0\}\ni0$ nonempty; by the Krasnosel'skii--Mann theorem for averaged operators \cite[Thm.~5.15]{bauschke2017convex}, $\Rz^n h$ converges to a fixed point, i.e., $\Rz^n h\to0$ for every $h$ (weak equals strong in finite dimension).
Define $\varphi_n(h):=\|\Rz^n(h)\|_{\cM}$ on the compact $\cM$-unit sphere: each $\varphi_n$ is continuous ($\Rz$ is $1$-Lipschitz), the sequence is pointwise nonincreasing (Step 3), and it converges pointwise to the continuous limit $0$; by Dini's theorem the convergence is uniform, so there is $r_*$ with $q_{r_*}=\max_{\|h\|_{\cM}=1}\varphi_{r_*}(h)<1$ (the maximum is attained by compactness).
By positive homogeneity,
\[
q_r:=\sup_{h\neq0}\frac{\|\Rz^r(h)\|_{\cM}}{\|h\|_{\cM}}\le1
\quad (r\ge1),
\qquad
q_{m+n}\le q_mq_n .
\]
Consequently,
\[
\|\Rz^n(h)\|_{\cM}
\le q_{r_*}^{\lfloor n/r_*\rfloor}\|h\|_{\cM}.
\]
Thus the state-direction gauge $q_r$ has the same no-transient and grouped-decay properties as the powers of a linear nonexpansive map.
Only under a locally fixed active set, where $\Rz$ restricts on $V$ to the linear map $L_1$ (Proposition~\ref{prop:smooth_bridge}), does this state decay enter the gradient-bias bound of Corollary~\ref{cor:finite_depth}.
When the active set changes, no gradient-bias statement follows.
This proves part (3).
\qed

\begin{remark}[Scope of the kernel mechanism]
\label{rmk:kernel_mechanism_scope}
The kernel argument uses strong convexity, density interiority, and the injective temporal block $D_t^\top$: it is specific to discretizations whose equality block contains an injective sub-block acting on coordinates that are never active.
This mechanism is absent, for example, in a periodic-in-time discretization, where $D_t$ is circulant.
Whether the conclusion of Theorem~\ref{thm:nonsmooth_kernel} itself persists in that setting is not addressed here.
\end{remark}

\section{Proofs of the Fixed-Active-Set Consistency Results}
\label{app:consistency_proof}

Throughout this section we retain the general resolvent family $R_{\theta,\lambda}$, $\lambda>0$, and the associated linearizations
\[
L_\lambda:=(\cM+\lambda H)^{-1}\cM,
\qquad
P_\lambda:=-\lambda(\cM+\lambda H)^{-1}B,
\]
which are well defined for every $\lambda>0$ because the symmetric part of $\cM+\lambda H$ is $\succeq\cM\succ0$.
The main article states its results at the unit resolvent parameter $\lambda=1$ only; in that notation $L=L_1$ and $P=P_1$, and $T_\theta=R_{\theta,1}$ by Theorem~\ref{thm:pdhg_resolvent}.
Every statement below specializes to the main-text form at $\lambda=1$.

The main article suppresses the coordinate projections of the reduced active-set representation.
The following proposition records them and identifies the reduction of the directional objects of section~\ref{subsec:nonsmooth}.

\begin{proposition}[Reduced coordinates and the fixed-active-set derivative resolvent]
\label{prop:smooth_bridge}
Suppose Assumptions~\ref{asm:pdhg_stepsizes}, \ref{asm:forward_regularity}, and~\ref{asm:local_chart} hold, and let $V:=\mathcal{K}\times\mathbb{R}^{N_c}$, where under a locally fixed active set the critical cone $\mathcal{K}$ is the subspace of primal directions vanishing at the strongly active slots, so that $V$ is the retained-coordinate subspace of Assumption~\ref{asm:local_chart}.
Then $\operatorname{ran}\Rz\subseteq V$, and for every $h\in V$, in the reduced coordinates, the graphical derivative~\eqref{eq:DA_def} is linear,
\[
D\cA(h)=Hh
\qquad\text{with $H$ as in~\eqref{eq:H_block}},
\]
and the derivative resolvent reduces to the linearized PDHG map $\Rz(h)=L_1h$.
In particular $V$ is $\Rz$-invariant.
Outside $V$, active coupling can give $\Rz(h)\ne L_1h$; all downstream uses lie in $V$, since $u_0=0$ and $\operatorname{ran}P_\lambda\subseteq V$.
\end{proposition}

\medskip
\noindent
\textbf{Reduced coordinates and proof of Proposition~\ref{prop:smooth_bridge}.} Let $\Pi_A:\mathbb{R}^{n_A}\to\mathbb{R}^n$ inject the retained primal coordinates and set $J_A:=\operatorname{diag}(\Pi_A,I_{N_c})$.
The reduced Hessian, constraint, and metric blocks are $\Pi_A^\top\nabla^2\cE_{\theta_0}\Pi_A$, $K\Pi_A$, and $J_A^\top\cM J_A$, respectively; the main article suppresses these projections and reuses $x$, $z$, $H$, $K$, and $\cM$.
Since the density is interior, all density columns are retained, so $K\Pi_A$ contains the unchanged $D_t$ block and has full row rank (appendix~\ref{app:K_structure}).

Under the locally fixed active set, strict complementarity makes the critical cone $\mathcal K=\operatorname{ran}\Pi_A$, a subspace.
If $s=\Rz(h)$, then $s_x\in\operatorname{dom}N_{\mathcal K}=\mathcal K$, so $\operatorname{ran}\Rz\subseteq V$.
For $h\in V$, the active components of $s_x$ and $h_x$ vanish; $N_{\{0\}}(0)=\mathbb R$ absorbs the active primal rows, while the inactive normal-cone factors vanish.
Writing $h=J_A\bar h$ and $s=J_A\bar s$, the defining inclusion therefore reduces to $(J_A^\top\cM J_A)(\bar s-\bar h)+H\bar s=0$.
Thus, in the suppressed reduced notation, $D\cA(h)=Hh$ and $\Rz(h)=(\cM+H)^{-1}\cM h=L_1h$; the inverse exists because the symmetric part of $\cM+H$ is $\succeq\cM\succ0$.

For $h\notin V$, the dual row is $S^{-1}(s_\phi-h_\phi)=Kh_x$ and retains the omitted coupling $K_{:,\mathrm{act}}h_{x,\mathrm{act}}$, so $\Rz(h)=L_1h$ need not hold.
Finally, differentiating the reduced resolvent equation gives $P_\lambda=-\lambda(\cM+\lambda H)^{-1}B$ in reduced coordinates; lifting each column through $J_A$ leaves every active primal component zero, proving $\operatorname{ran}P_\lambda\subseteq V$.
The same reduced construction gives $L_\lambda(V)\subseteq V$.
Since $u_0=0$ and $u_{j+1}=L_\lambda u_j+P_\lambda$, every downstream use remains in $V$.
\hfill$\square$

\medskip
\noindent
\textbf{Proof of the spectral claim of Proposition~\ref{prop:side_by_side}.} Let $L_\lambda v=\mu v$ with $v\neq0$.
If $\mu=1$, then $(\cM+\lambda H)v=\cM v$, hence $Hv=0$; conversely $Hv=0$ implies $L_\lambda v=v$.
Thus $\operatorname{Fix}(L_\lambda)=\ker H$.
Since $L_\lambda$ is invertible, $\mu\neq0$, and the eigenvalue equation rearranges to
\[
Hv=c\,\cM v,
\qquad
c=\frac{1-\mu}{\mu\lambda},
\qquad
\mu=(1+\lambda c)^{-1}.
\]
Multiplication by $v^*$ gives $v^*Hv=c\,v^*\cM v$.
Monotonicity of $H$ implies $\operatorname{Re}(c)\ge0$, and hence $|\mu|=|1+\lambda c|^{-1}\le1$, with equality only when $c=0$.
Therefore $\spr(L_\lambda)\le1$ and $\spr(L_\lambda)<1$ exactly when $\ker H=\{0\}$.

Under the assumptions, $H$ is invertible by Lemma~\ref{lem:H_invertible}. Hence $\ker H=\{0\}$, and therefore $\spr(L_\lambda)<1$.
\hfill$\square$

\medskip
\noindent
\textbf{Proof of Theorem~\ref{thm:jfb_consistency}.} Fix $\lambda>0$.
Exact detachment gives $z^j=z_0$ for every tracked step and $u_0=0$.
Differentiating $z^{j+1}=R_{\theta,\lambda}(z^j)$ and the local resolvent equation gives
\[
u_{j+1}=L_\lambda u_j+P_\lambda,
\qquad
u_r=\sum_{k=0}^{r-1}L_\lambda^kP_\lambda.
\]
By the general-$\lambda$ spectral argument above, $\spr(L_\lambda)<1$, and hence
\[
u_r=(I-L_\lambda)^{-1}(I-L_\lambda^r)P_\lambda
\longrightarrow (I-L_\lambda)^{-1}P_\lambda=-H^{-1}B,
\]
where the last identity follows directly from the definitions of $L_\lambda$ and $P_\lambda$.
Moreover,
\[
u_r+H^{-1}B=-(I-L_\lambda)^{-1}L_\lambda^rP_\lambda.
\]
This is the implicit sensitivity of Lemma~\ref{lem:H_invertible} and its finite-depth bias.
For $\lambda=1$, pairing with $g$ proves the stated convergence and bound.
If the true gradient is nonzero, the bound tends to zero; once it falls below $\|\nabla_\theta F(\theta_0)\|$, the Cauchy--Schwarz estimate $\langle a,d\rangle\ge\|a\|^2-\|a\|\,\|d-a\|$ with $a=\nabla_\theta F(\theta_0)$ and $d=d_{r}^{\mathrm{JFB}}(\theta_0)$ gives eventual descent.
\hfill$\square$

\medskip
Because $\lambda>0$ was arbitrary, the argument covers the whole resolvent family $R_{\theta,\lambda}$; the main article states only the case $\lambda=1$, for which $R_{\theta,1}=T_\theta$.

The following conditional bound is the finite-depth counterpart of Theorem~\ref{thm:jfb_consistency}; the grouped depth $r_*$ enters only here, converting an operator-norm estimate into a quantitative certificate at finite tracked depth.
Its operator-tail hypothesis is not verified in the reported experiments.

\begin{corollary}[Conditional finite-depth error bound]
\label{cor:finite_depth}
Under the assumptions of Theorem~\ref{thm:jfb_consistency}, one has $\|L_1^{\,k}\|_{\cM}\le1$ for all $k\ge0$.
Suppose a (numerically verified or proved) grouped-map upper bound
\[
\|L_1^{\,r_*}\|_{\cM}\;\le\;q_{\cM}\;<\;1
\qquad\text{for some } r_*\ge1
\]
is available.
Then $\spr(L_1)\le q_{\cM}^{1/r_*}<1$, and for every integer $r\ge1$,
\[
\begin{aligned}
\|L_1^{\,r}\|_{\cM}
&\le q_{\cM}^{\lfloor r/r_*\rfloor},\\
\bigl\|d_{r}^{\mathrm{JFB}}(\theta_0)
-\nabla_\theta F(\theta_0)\bigr\|_2
&\le
C_{\cM}\,
\frac{r_*\,q_{\cM}^{\lfloor r/r_*\rfloor}}{1-q_{\cM}},
\end{aligned}
\]
where $C_{\cM}:=\|g\|_{\cM^{-1}}\,\|P_1\|_{2\to\cM}$.
If moreover the right-hand side is smaller than $\|\nabla_\theta F(\theta_0)\|_2$, then $\langle\nabla_\theta F(\theta_0),d_{r}^{\mathrm{JFB}}(\theta_0)\rangle>0$.
\end{corollary}

Such a pair $(r_*,q_{\cM})$ exists: Theorem~\ref{thm:nonsmooth_kernel}(3) together with Proposition~\ref{prop:smooth_bridge} gives $\spr(L_1)<1$ on $V$, and Gelfand's formula then supplies a finite $r_*$ with $q_{\cM}<1$, though without an explicit estimate of $r_*$.

Because $L_1$ is generally non-normal, $\spr(L_1)^r$ alone is \emph{not} a valid finite-depth bound: $\|L_1^r\|$ can grow transiently before decaying, and no a priori bound on the required depth $r_0$ of Theorem~\ref{thm:jfb_consistency} is claimed.
The $\cM$-norm is used because firm nonexpansiveness removes that transient.

\medskip
\noindent
\textbf{Proof of Corollary~\ref{cor:finite_depth}.} Firm nonexpansiveness and differentiability at the fixed point give $\|L_1\|_{\cM}\le1$.
Write $r=jr_*+s$ with $0\le s<r_*$.
Then
\[
\|L_1^r\|_{\cM}
\le\|L_1^{r_*}\|_{\cM}^j\|L_1^s\|_{\cM}
\le q_{\cM}^{\lfloor r/r_*\rfloor}.
\]
Also,
\[
\|(I-L_1)^{-1}\|_{\cM}
\le\sum_{k\ge0}\|L_1^k\|_{\cM}
\le\frac{r_*}{1-q_{\cM}}.
\]
For any $\xi\in\mathbb{R}^p$ with $\|\xi\|_2=1$,
\[
\begin{aligned}
\bigl|\langle d_{r}^{\mathrm{JFB}}-\nabla_\theta F,\xi\rangle\bigr|
&=
\bigl|\langle (I-L_1)^{-1}L_1^rP_1\xi,\cM^{-1}g\rangle_{\cM}\bigr|\\
&\le
\|(I-L_1)^{-1}\|_{\cM}\|L_1^r\|_{\cM}
\|P_1\|_{2\to\cM}\|g\|_{\cM^{-1}}.
\end{aligned}
\]
Combining the displays gives the stated bound, and the Cauchy--Schwarz estimate from the proof of Theorem~\ref{thm:jfb_consistency} gives the descent conclusion.
\hfill$\square$

\section{Proof of Proposition~\ref{prop:adjoint_alignment}}
\label{app:adjoint_alignment}

\begin{proposition}[Adjoint-space alignment before parameter compression]
\label{prop:adjoint_alignment}
Suppose Assumptions~\ref{asm:pdhg_stepsizes}, \ref{asm:forward_regularity}, and~\ref{asm:local_chart} hold at $(\theta_0,z_0)$ with exact equilibrium detachment $\bar z=z_0=z_{\theta_0}^\star$, let $H$, $B$ be as in~\eqref{eq:H_block} and $g:=\nabla_z\widehat\ell(z_0)\neq0$, and fix $\lambda>0$.
For $r\ge1$, write
\[
d_{r}^{\mathrm{JFB},\lambda}(\theta_0)
:=\bigl(\partial_\theta R_{\theta,\lambda}^{\,r}(z_0)\big|_{\theta=\theta_0}\bigr)^{\!\top} g
\]
for the JFB-$r$ direction built from the resolvent family $R_{\theta,\lambda}$; at $\lambda=1$ this is the $d_r^{\mathrm{JFB}}$ of the main article.
Define the adjoint iterates $w_0:=0$ and
\[
\cM\,(w_{k+1}-w_k)+\lambda\bigl(H^\top w_{k+1}-g\bigr)=0,
\qquad k\ge0 ,
\]
the general-$\lambda$ form of~\eqref{eq:adjoint_pp}.
Then:
\begin{enumerate}
\item \emph{(Adjoint iteration)} Each $w_{k+1}$ is well defined, and the recursion is the proximal-point (backward Euler) iteration, in the metric $\cM$ with parameter $\lambda$, for the adjoint system $H^\top w=g$ associated with the implicit-gradient formula~\eqref{eq:implicit_gradient}.
Its unique fixed point is $w^\star:=H^{-\top}g$, and
    \[
    d_{r}^{\mathrm{JFB},\lambda}(\theta_0)=-B^\top w_r
    \quad\text{for every } r\ge1,
    \qquad
    \nabla_\theta F(\theta_0)=-B^\top w^\star .
    \]

\item \emph{($\cM$-alignment)} For every $r\ge1$ and every $\lambda>0$,
    \begin{equation}
    2\,\langle w^\star,\,w_r\rangle_{\cM}
    \;\ge\;
    \|w_r\|_{\cM}^2+\sum_{k=0}^{r-1}\|w_{k+1}-w_k\|_{\cM}^2
    \;>\;0 ,
    \label{eq:adjoint_alignment}
    \end{equation}
and $\|w_r\|_{\cM}\le2\|w^\star\|_{\cM}$.

\item \emph{(Transfer identity)}
    \[
    \bigl\langle \nabla_\theta F(\theta_0),\,
    d_{r}^{\mathrm{JFB},\lambda}(\theta_0)\bigr\rangle
    \;=\;
    \bigl\langle w^\star,\;BB^\top w_r\bigr\rangle ,
    \]
the pairing on the right being the Euclidean primal-dual pairing.
\end{enumerate}
\end{proposition}

Part~(2) is an alignment statement in the $\cM$-geometry of the adjoint space; by part~(3), descent for the true objective can be lost only in the final compression through $BB^\top$, which is what the parameter-space alignment of Experiment~\ref{sec:exp1} diagnoses.
Active-set changes, finite forward tolerance, and approximate detachment lie outside this model.

\begin{proof}
\emph{(1)} The symmetric part of $\cM+\lambda H^\top$ is $\succeq\cM\succ0$ by monotonicity of $H$, so the step is well defined and \eqref{eq:adjoint_pp} solves to $w_{k+1}=\widetilde L_\lambda w_k+\lambda(\cM+\lambda H^\top)^{-1}g$ with $\widetilde L_\lambda:=(\cM+\lambda H^\top)^{-1}\cM$; started at $w_0=0$ this gives $w_r=\sum_{k=0}^{r-1}\widetilde L_\lambda^{\,k}\,\lambda \widetilde C_\lambda g$ with $\widetilde C_\lambda:=(\cM+\lambda H^\top)^{-1}$.
Transposing one term of~\eqref{eq:jfb_closed_form}, with $L_\lambda=(\cM+\lambda H)^{-1}\cM$ and $P_\lambda=-\lambda(\cM+\lambda H)^{-1}B$,
\[
\bigl(L_\lambda^{\,k}P_\lambda\bigr)^\top
=-\lambda\,B^\top \widetilde C_\lambda\,(\cM \widetilde C_\lambda)^k
=-\lambda\,B^\top (\widetilde C_\lambda\cM)^k \widetilde C_\lambda
=-B^\top\,\widetilde L_\lambda^{\,k}\,\lambda \widetilde C_\lambda ,
\]
where the middle equality is associativity, $\widetilde C_\lambda(\cM \widetilde C_\lambda)^k=(\widetilde C_\lambda\cM)^k\widetilde C_\lambda$; summing over $k$ gives $d_{r}^{\mathrm{JFB},\lambda} =\sum_{k=0}^{r-1}(L_\lambda^{\,k}P_\lambda)^\top g=-B^\top w_r$.
A fixed point of~\eqref{eq:adjoint_pp} satisfies $H^\top w=g$; under the assumptions $H$ is invertible (Proposition~\ref{prop:side_by_side}), so $w^\star=H^{-\top}g$ is unique, and $\nabla_\theta F(\theta_0)=-B^\top H^{-\top}g =-B^\top w^\star$ by~\eqref{eq:implicit_gradient}.
Subtracting the fixed-point equation from the recursion, $w_{k+1}-w^\star=\widetilde L_\lambda(w_k-w^\star)$.
Since $\cM\widetilde L_\lambda\cM^{-1}=L_\lambda^\top$, we have $\spr(\widetilde L_\lambda)=\spr(L_\lambda)<1$, and hence $w_r\to w^\star$.

\emph{(2)} For arbitrary $v$ set $u:=\widetilde L_\lambda v$, so $v=\cM^{-1}(\cM+\lambda H^\top)u$ and
\[
\bigl\langle \widetilde L_\lambda v,\,
(I-\widetilde L_\lambda)v\bigr\rangle_{\cM}
= u^\top\cM(v-u)
= \lambda\,u^\top H^\top u
\;\ge\;0
\]
by monotonicity of $H$ (display~\eqref{eq:H_block}).
Expanding $\|v\|_{\cM}^2 =\|\widetilde L_\lambda v+(I-\widetilde L_\lambda)v\|_{\cM}^2$ and discarding the nonnegative cross term, $\|\widetilde L_\lambda v\|_{\cM}^2 +\|(I-\widetilde L_\lambda)v\|_{\cM}^2\le\|v\|_{\cM}^2$.
Applying this with $v=w_k-w^\star$, for which $\widetilde L_\lambda v=w_{k+1}-w^\star$ and $(I-\widetilde L_\lambda)v=w_k-w_{k+1}$, yields the Fej\'er inequality $\|w_{k+1}-w^\star\|_{\cM}^2 \le\|w_k-w^\star\|_{\cM}^2-\|w_{k+1}-w_k\|_{\cM}^2$; telescoping from $\|w_0-w^\star\|_{\cM}^2=\|w^\star\|_{\cM}^2$ and polarizing $2\langle w^\star,w_r\rangle_{\cM} =\|w^\star\|_{\cM}^2+\|w_r\|_{\cM}^2-\|w_r-w^\star\|_{\cM}^2$ gives~\eqref{eq:adjoint_alignment}.
Positivity is strict because the $k=0$ term equals $\|w_1\|_{\cM}^2$ and $w_1=\lambda(\cM+\lambda H^\top)^{-1}g\neq0$ for $g\neq0$.
Finally $\|w_r\|_{\cM}\le\|w^\star\|_{\cM}+\|w_r-w^\star\|_{\cM} \le2\|w^\star\|_{\cM}$ by the telescoped Fej\'er bound.

\emph{(3)} Combine the two identities of part~(1).
\end{proof}

\medskip
\noindent
The arguments of appendices~\ref{app:pdhg_resolvent_proof}--\ref{app:adjoint_alignment} extend to other linearly constrained convex problems whenever the corresponding metric positivity, local regularity, and reduced-Jacobian conditions hold.

\section{Experimental details and secondary diagnostics}
\label{app:experimental_details}

This section records the protocols and secondary checks summarized in
section~\ref{sec:experiments} of the main article. Code, exact
configurations, seeds, and diagnostic scripts will be released at
\url{https://github.com/siting6/InverseMFG_JFB}.

\subsection{Common protocol}

All reconstruction runs use dual-extrapolated G-prox PDHG with $W=KK^\top$,
$\tau=\sigma=0.99$, and $\gamma_I=0.01$. Outer iterations reuse the
previous detached primal-dual state; the initial outer iteration uses
a cold start in which the density block tiles $\mu_0$ across all time
steps and the flux and dual blocks are zero. Reference observations use the same solver structure at the ground-truth cost;
the Experiment 4 reference-data settings are given below.

The learned cost is
\[
c_\theta(s)=0.01+\operatorname{softplus}(\operatorname{MLP}_\theta(s)),
\]
where the three-layer tanh MLP uses Flax's default initialization and
no pretraining. Experiments~1--3 use the normalized density pairs
\[
\mu_0(s)\propto1.25-0.25\cos(4\pi s),\qquad
\mu_1(s)\propto1.25+0.25\cos(2\pi s)
\]
in 1D and
\[
\begin{aligned}
\mu_0(s_1,s_2)&\propto
1.25-0.125\cos(4\pi s_1)-0.125\cos(4\pi s_2),\\
\mu_1(s_1,s_2)&\propto
1.25+0.125\cos(2\pi s_1)+0.125\cos(2\pi s_2)
\end{aligned}
\]
in 2D.

Adam uses $\beta_2=0.999$, gradient clipping at norm one, and a cosine
learning-rate schedule from $10^{-2}$ to $10^{-3}$, with
$\beta_1=0.5$ in Experiments~1--3 and $0.9$ in Experiment~4. Methods
use matched data realizations. Computations use float64 JAX, Flax, and
Optax on one 48~GB NVIDIA RTX~6000 Ada; peak memory is measured in a
fresh process. Checkpoint selection and relative error follow the
definitions in section~\ref{sec:experiments}.
Table~\ref{tab:experimental_configs} collects the per-experiment
configurations.

\begin{table}[htbp]
\centering
\footnotesize
\caption{Common experimental configurations. All runs use
$N_{\mathrm{itr}}=100$ PDHG steps per outer iteration.}
\label{tab:experimental_configs}
\begin{tabular}{lccccccc}
\toprule
{Exp.} & {Grid/$N_t$} & {Width} & {Outer} & {$\gamma_T$}
& {Observations/noise} & {Seeds} & {Batch} \\
\midrule
1 & $64/16$ & 32 & 500 & 0.5 & full/clean & 3 & full \\
2 & $32^2/16$--$128^2/64$ & 64 & 300 & 0.1 & full/clean & $1$--$3$ & full \\
3 & $32^2/16$ & 64 & 1000 & 0.1 & full or 50\%/$\gamma_n\le0.05$ & 3 & full \\
4 & $128^2/64\to64^2/32$ & 64 & 2000 & 0.1 & 25\%/$\gamma_n=0.2$ & 3 & 2 \\
\bottomrule
\end{tabular}
\end{table}

\subsection{Experiment 1 implementation and ablation}
\label{app:exp1_details}

AD differentiates all $100$ PDHG steps, whereas AD-ckpt
rematerializes ten segments of ten steps and matches AD loss and
gradients within $2\times10^{-10}$. ID applies BiCGSTAB to
$J_\theta^\top v=\nabla_z\widehat\ell$, where
$J_\theta:=I-D_z\widetilde T_\theta(z)$ is the fixed-point Jacobian of the
implementation-selected branch at the returned state, with tolerance
$10^{-9}$ and at most $200$ iterations. It reuses the matched detached
state rather than re-solving the forward problem inside the custom VJP,
which is the favorable variant for implicit differentiation under the
shared warm-started protocol. JFB-$r$ records only the final $r$ steps.

\subsubsection{Tracked-depth ablation}
\label{app:exp1_ablation}

At $r=1$ the method fails under both learning rates tested. At the
default rate $10^{-2}$ it diverges, with a final error that is itself
unstable, ranging from $141\%$ to $151\%$ across three seeds
at fixed software version; reducing the rate to
$10^{-5}$ prevents divergence but stalls reproducibly near $30\%$
error ($30.3\pm0.9\%$ across the same seeds). Recovery for
$r=5,10$ and for full AD is reported in Table~\ref{tab:exp1} of the
main article.

\subsubsection{Alignment protocol}
\label{app:exp1_alignment_details}

Three JFB-10 runs provide $21$ checkpoints each, spaced every $25$
outer iterations. At each checkpoint,
$r\in\{1,5,10,20\}$ tracks the final steps of the warm solve. The
cold-started ID reference uses $1500$ forward iterations and BiCGSTAB
tolerance $10^{-9}$; warm AD satisfies
$\cos(d^{\mathrm{AD}},d^{\mathrm{ID\text{-}ref}})\ge0.99$ at all $63$
checkpoints. The five earliest checkpoints of each run, those at outer
iterations at most $100$, form the $15$ descent-phase checkpoints; the
rest lie on a plateau where the reference norm is about two orders of
magnitude smaller. All directions are evaluated at states visited by
the JFB-10 runs, so the $r\ne10$ columns are retrospective.

Table~\ref{tab:alignment_all} reports the statistics over all $63$
checkpoints, of which Table~\ref{tab:exp1_alignment} of the main
article is the descent-phase restriction. Because the plateau
reference norm is small, the $r=5$ entries degrade there --- its
full-set median cosine is negative --- while $r=1$ improves; the depth
ordering by median cosine is non-monotone on both checkpoint sets,
with $r=10$ above $r=20$ on the descent phase and below it over all
$63$. Over the $15$
descent-phase checkpoints of Table~\ref{tab:exp1_alignment} of the main
article, the minimum cosines are $-1.00$, $-0.48$, $+0.50$, and $+0.55$
at $r=1,5,10,20$: the two moderate depths are positively aligned at
every descent-phase checkpoint.

Tables~\ref{tab:exp1_alignment} and~\ref{tab:alignment_all} are
computed from the primary alignment sweep; the detach residuals of
appendix~\ref{app:extended_diag} come from a later rerun of the same
sweep under the identical configuration, and the active-set statistics
from a separate equilibrium diagnostic at the ground-truth cost. The
depth ordering by median cosine is unchanged across the two sweeps.

\begin{table}[htbp]
\centering
\caption{Experiment 1 alignment over all $63$ checkpoints (three
JFB-10 trajectories, $21$ checkpoints each). Cosine and relative bias
are Euclidean in parameter space, against
$d^{\mathrm{ID\text{-}ref}}$.}
\label{tab:alignment_all}
\begin{tabular}{rrrrr}
\toprule
$r$ & median $\cos$ & min $\cos$ & positive checkpoints & med.\ rel.\ bias \\
\midrule
$1$  & $+0.21$ & $-1.00$ & $36/63$ & $0.99$ \\
$5$  & $-0.53$ & $-0.88$ & $23/63$ & $1.31$ \\
$10$ & $+0.63$ & $-0.33$ & $55/63$ & $0.96$ \\
$20$ & $+0.89$ & $+0.47$ & $63/63$ & $0.53$ \\
\bottomrule
\end{tabular}
\end{table}

\subsection{Experiment 2 robustness checks}
\label{app:exp2_checks}

At $32^2$, three-seed errors are stable (AD
$1.41/1.48/1.45\%$; JFB-10 $2.03/2.01/1.92\%$). At $64^2$, all
methods show occasional late instability after the loss reaches
$10^{-8}$ (AD $1.43/4.95/1.63\%$; JFB-10
$2.36/72.19/2.30\%$; AD-ckpt $152.6/8.08/1.72\%$). AD-ckpt and AD
compute the same direction in exact arithmetic; the divergent $64^2$
triple occurs in the same late, post-saturation regime in which all
methods here show instability, and reflects float64 arithmetic rather
than a different derivative. One AD run has
$47.6\%$ cost error despite loss $8.5\times10^{-9}$, so low data
misfit need not imply recovery. Increasing the $128^2$ forward budget
from $100$ to $400$ steps gives $3.02\pm0.27\%$ error versus $2.92\%$
in the primary run, so the reported $128^2$ comparison is not
forward-budget limited.

\subsection{Experiment 3 data and noise}
\label{app:exp3_details}

Experiment~3 uses independent noise
$\rho^{\mathrm{obs}}=\rho^{\mathrm{clean}}+\gamma_nU[-0.5,0.5]$,
clamped below at $10^{-8}$ without renormalization. The $50\%$ setting
observes eight of sixteen output times, excluding the fixed initial
density; seeds determine initialization, mask, and noise.

\subsection{Experiment 4 construction}
\label{app:exp4_details}

Ten instances share the ground-truth map $c^\star=c_{\mathrm{mix}}$,
where
\[
\begin{aligned}
c_{\mathrm{mix}}(s_1,s_2)
&=0.8-0.12\cos(2\pi s_1)-0.12\cos(2\pi s_2)\\
&\quad+1.2\exp\!\left(
-\frac{(s_1+0.18)^2+(s_2-0.12)^2}{2(0.10)^2}\right)\\
&\quad+1.0\exp\!\left(
-\frac{(s_1-0.20)^2+(s_2+0.16)^2}{2(0.12)^2}\right).
\end{aligned}
\]
Each initial and terminal density is a unit-mass two-Gaussian mixture
with a positive background. Centers lie in the square ring
$[-0.4,0.4]^2\setminus(-0.25,0.25)^2$ and initial/terminal peaks
occupy opposite annular sectors, forcing transport through the central
inverse-cost bumps. Reference trajectories are computed on a $128^2$, $N_t=64$ grid with
$\tau=\sigma=0.95$ and a $2\times10^4$-iteration budget, then restricted
to $64^2$, $N_t=32$ by $2\times2$ spatial averaging and stride-two time
subsampling. Reconstruction uses analytic coarse-grid endpoint densities.
Each instance observes eight of $32$ output times with noise $\gamma_n=0.2$;
masks and noise are applied after restriction and fixed across methods and seeds.

The instances are fixed; training seeds control initialization and
mini-batch sampling. The resulting reconstructions are shown in
Figure~\ref{fig:exp4_recovery} of the main article.

\subsection{Alignment probe for Experiment 3}
\label{app:exp3_alignment}

This probe is a single representative noisy, partial-in-time
Experiment~3 run, evaluated at $21$ checkpoints (outer iterations
$0$, $50$, \ldots, $950$, $999$). The reported cosines compare JFB-$r$ with the
warm-started full-AD direction at the same checkpoint, not with the
Experiment~1 implicit reference:
\[
\begin{array}{c|rrrr}
r & 1 & 5 & 10 & 25\\ \hline
\text{mean }\cos & -0.08 & -0.40 & 0.35 & 0.99\\
\text{positive} & 33\% & 24\% & 81\% & 100\%
\end{array}
\]
On this run the warm-AD reference has cosine at least $0.959$ with a
separately cold-started AD cross-check at the same
$N_{\mathrm{itr}}=100$. The depth ordering matches the $r=10$ versus $r=25$ recovery
difference in Table~\ref{tab:exp3}.

\subsection{Diagnostics for the theoretical assumptions}
\label{app:extended_diag}

\paragraph{Active-set behavior}
At the true-cost equilibrium, whose fixed-point residual is $4.5\times10^{-13}$, density lies in $[0.80,1.20]$ and flux magnitude is at most $4.3\times10^{-2}$.
Each flux block has $512$ of $1008$ slots on the boundary: $496$ are strongly active and $16$ are degenerate at multiplier tolerance $10^{-8}$.
Nonzero boundary multipliers have median $2.8\times10^{-2}$ and maximum $7.0\times10^{-2}$, against a noise floor near $8\times10^{-14}$.
Across $21$ checkpoints per seed, active sets change in bursts of up to $48$ slots during descent and freeze only on the late plateau.
Thus no single fixed active set covers the trajectory.

\paragraph{Detach residuals}
After the untracked warm-up, the implemented detach residual $\|\widetilde T_\theta(z)-z\|_{\cM}$ has median $2\times10^{-7}$ and maximum $1.2\times10^{-5}$ over all seeds and depths, excluding the initial outer iteration, which begins without a cached warm state; the state norm is about $35$.
The incoming cached-state residual has median $5\times10^{-5}$ and early maximum $6.2\times10^{-2}$.
These residuals diagnose proximity to a fixed point but do not bound sensitivity error.

\bibliographystyle{siamplain}
\bibliography{references}

\begin{thebibliography}{10}

\bibitem{achdou2012planning}
{\sc Y.~Achdou, F.~Camilli, and I.~Capuzzo-Dolcetta}, {\em Mean field games:
  Numerical methods for the planning problem}, SIAM Journal on Control and
  Optimization, 50 (2012), pp.~77--109,
  \url{https://doi.org/10.1137/100790069}.

\bibitem{achdou2010mean}
{\sc Y.~Achdou and I.~Capuzzo-Dolcetta}, {\em Mean field games: numerical
  methods}, SIAM Journal on Numerical Analysis, 48 (2010), pp.~1136--1162.

\bibitem{achdou2020mean}
{\sc Y.~Achdou and M.~Lauri\`{e}re}, {\em Mean field games and applications:
  numerical aspects}, in Mean Field Games, vol.~2281 of Lecture Notes in
  Mathematics, Springer, Cham, 2020, pp.~249--307.

\bibitem{agrawal2022random}
{\sc S.~Agrawal, W.~Lee, S.~Wu~Fung, and L.~Nurbekyan}, {\em Random features
  for high-dimensional nonlocal mean-field games}, Journal of Computational
  Physics, 459 (2022), p.~111136.

\bibitem{andrade2024sparsistency}
{\sc F.~Andrade, G.~Peyr{\'e}, and C.~Poon}, {\em Sparsistency for inverse
  optimal transport}, in International Conference on Learning Representations,
  vol.~2024, 2024, pp.~16575--16606.

\bibitem{bai2019deep}
{\sc S.~Bai, J.~Z. Kolter, and V.~Koltun}, {\em Deep equilibrium models}, in
  Proceedings of the 33rd International Conference on Neural Information
  Processing Systems, 2019, pp.~690--701.

\bibitem{bauschke2017convex}
{\sc H.~H. Bauschke and P.~L. Combettes}, {\em Convex Analysis and Monotone
  Operator Theory in Hilbert Spaces}, CMS Books in Mathematics, Springer, Cham,
  2nd~ed., 2017.

\bibitem{benamou2000computational}
{\sc J.-D. Benamou and Y.~Brenier}, {\em A computational fluid mechanics
  solution to the {M}onge-{K}antorovich mass transfer problem}, Numerische
  Mathematik, 84 (2000), pp.~375--393.

\bibitem{benamou2015augmented}
{\sc J.-D. Benamou and G.~Carlier}, {\em Augmented {L}agrangian methods for
  transport optimization, mean field games, and degenerate elliptic equations},
  Journal of Optimization Theory and Applications, 167 (2015), pp.~1--26.

\bibitem{benamou2017variational}
{\sc J.-D. Benamou, G.~Carlier, and F.~Santambrogio}, {\em Variational mean
  field games}, in Active Particles, Volume 1: Advances in Theory, Models, and
  Applications, Birkh\"{a}user, Cham, 2017, pp.~141--171.

\bibitem{bricenoarias2019implementation}
{\sc L.~M. Brice{\~n}o-Arias, D.~Kalise, Z.~Kobeissi, M.~Lauri\`{e}re,
  A.~Mateos~Gonz\'{a}lez, and F.~J. Silva}, {\em On the implementation of a
  primal-dual algorithm for second order time-dependent mean field games with
  local couplings}, ESAIM: Proceedings and Surveys, 65 (2019), pp.~330--348.

\bibitem{bricenoarias2018proximal}
{\sc L.~M. Brice{\~n}o-Arias, D.~Kalise, and F.~J. Silva}, {\em Proximal
  methods for stationary mean field games with local couplings}, SIAM Journal
  on Control and Optimization, 56 (2018), pp.~801--836,
  \url{https://doi.org/10.1137/16M1095615}.

\bibitem{chambolle2011firstorder}
{\sc A.~Chambolle and T.~Pock}, {\em A first-order primal-dual algorithm for
  convex problems with applications to imaging}, Journal of Mathematical
  Imaging and Vision, 40 (2011), pp.~120--145,
  \url{https://doi.org/10.1007/s10851-010-0251-1}.

\bibitem{chambolle2016ergodic}
{\sc A.~Chambolle and T.~Pock}, {\em On the ergodic convergence rates of a
  first-order primal--dual algorithm}, Mathematical Programming, 159 (2016),
  pp.~253--287, \url{https://doi.org/10.1007/s10107-015-0957-3}.

\bibitem{chow2023numerical}
{\sc Y.~T. Chow, S.~Wu~Fung, S.~Liu, L.~Nurbekyan, and S.~Osher}, {\em A
  numerical algorithm for inverse problem from partial boundary measurement
  arising from mean field game problem}, Inverse Problems, 39 (2023),
  p.~014001.

\bibitem{condat2013primal}
{\sc L.~Condat}, {\em A primal--dual splitting method for convex optimization
  involving {L}ipschitzian, proximable and linear composite terms}, Journal of
  Optimization Theory and Applications, 158 (2013), pp.~460--479,
  \url{https://doi.org/10.1007/s10957-012-0245-9}.

\bibitem{daosud2005neural}
{\sc W.~Daosud, P.~Thitiyasook, A.~Arpornwichanop, P.~Kittisupakorn, and M.~A.
  Hussain}, {\em Neural network inverse model-based controller for the control
  of a steel pickling process}, Computers \& Chemical Engineering, 29 (2005),
  pp.~2110--2119.

\bibitem{ding2022mean}
{\sc L.~Ding, W.~Li, S.~Osher, and W.~Yin}, {\em A mean field game inverse
  problem}, Journal of Scientific Computing, 92 (2022), p.~7.

\bibitem{dontchev2014implicit}
{\sc A.~L. Dontchev and R.~T. Rockafellar}, {\em Implicit Functions and
  Solution Mappings: A View from Variational Analysis}, Springer, New York,
  2nd~ed., 2014.

\bibitem{el2021implicit}
{\sc L.~El~Ghaoui, F.~Gu, B.~Travacca, A.~Askari, and A.~Tsai}, {\em Implicit
  deep learning}, SIAM Journal on Mathematics of Data Science, 3 (2021),
  pp.~930--958.

\bibitem{gelphman2025end}
{\sc E.~Gelphman, D.~Verma, N.~T. Yang, S.~Osher, and S.~Wu~Fung}, {\em
  End-to-end training of high-dimensional optimal control with implicit
  {H}amiltonians via {J}acobian-free backpropagation}, arXiv preprint
  arXiv:2510.00359,  (2025).

\bibitem{gelphman2026convergence}
{\sc E.~Gelphman, D.~Verma, N.~T. Yang, S.~Osher, and S.~Wu~Fung}, {\em On the
  convergence of {J}acobian-free backpropagation for optimal control problems
  with implicit {H}amiltonians}, arXiv preprint arXiv:2602.00921,  (2026).

\bibitem{geng2021training}
{\sc Z.~Geng, X.-Y. Zhang, S.~Bai, Y.~Wang, and Z.~Lin}, {\em On training
  implicit models}, Advances in neural information processing systems, 34
  (2021), pp.~24247--24260.

\bibitem{gonzalez2024nonlinear}
{\sc A.~Gonz{\'a}lez-Sanz, M.~Groppe, and A.~Munk}, {\em Nonlinear inverse
  optimal transport: Identifiability of the transport cost from its marginals
  and optimal values}, SIAM Journal on Mathematical Analysis, 56 (2024),
  pp.~7808--7829.

\bibitem{guo2024decoding}
{\sc J.~Guo, C.~Mou, X.~Yang, and C.~Zhou}, {\em Decoding mean field games from
  population and environment observations by {G}aussian processes}, Journal of
  Computational Physics, 508 (2024), p.~112978,
  \url{https://doi.org/10.1016/j.jcp.2024.112978}.

\bibitem{he2012convergence}
{\sc B.~He and X.~Yuan}, {\em Convergence analysis of primal-dual algorithms
  for a saddle-point problem: From contraction perspective}, SIAM Journal on
  Imaging Sciences, 5 (2012), pp.~119--149.

\bibitem{heaton2023explainable}
{\sc H.~Heaton and S.~Wu~Fung}, {\em Explainable {AI} via learning to
  optimize}, Scientific Reports, 13 (2023), p.~10103.

\bibitem{heaton2021feasibility}
{\sc H.~Heaton, S.~Wu~Fung, A.~Gibali, and W.~Yin}, {\em Feasibility-based
  fixed point networks}, Fixed Point Theory and Algorithms for Sciences and
  Engineering, 2021 (2021), p.~21.

\bibitem{huang2025joint}
{\sc H.~Huang, J.~Yu, T.~Chen, and R.~Lai}, {\em Joint inference of trajectory
  and obstacle in mean-field games via bilevel optimization}, arXiv preprint
  arXiv:2507.19344,  (2025).

\bibitem{caines2006large}
{\sc M.~Huang, R.~P. Malham{\'e}, and P.~E. Caines}, {\em Large population
  stochastic dynamic games: closed-loop {McKean-Vlasov} systems and the {N}ash
  certainty equivalence principle}, Communications in Information and Systems,
  6 (2006), pp.~221--252.

\bibitem{imanuvilov2024lipschitz}
{\sc O.~Imanuvilov, H.~Liu, and M.~Yamamoto}, {\em Lipschitz stability for
  determination of states and inverse source problem for the mean field game
  equations}, Inverse Problems and Imaging, 18 (2024), pp.~824--859,
  \url{https://doi.org/10.3934/ipi.2023057}.

\bibitem{jacobs2019solving}
{\sc M.~Jacobs, F.~L{\'e}ger, W.~Li, and S.~Osher}, {\em Solving large-scale
  optimization problems with a convergence rate independent of grid size}, SIAM
  Journal on Numerical Analysis, 57 (2019), pp.~1100--1123.

\bibitem{knutson2026logical}
{\sc B.~Knutson, A.~C. Rabeendran, M.~Ivanitskiy, J.~Pettyjohn, C.~D. Behn,
  S.~Wu~Fung, and D.~McKenzie}, {\em On logical extrapolation for mazes with
  recurrent and implicit networks}, in Proceedings of the AAAI Conference on
  Artificial Intelligence, vol.~40, 2026, pp.~22635--22643.

\bibitem{lasry2007mean}
{\sc J.-M. Lasry and P.-L. Lions}, {\em Mean field games}, Japanese journal of
  mathematics, 2 (2007), pp.~229--260.

\bibitem{koltun2012continuous}
{\sc S.~Levine and V.~Koltun}, {\em Continuous inverse optimal control with
  locally optimal examples}, in International Conference on Machine Learning
  (ICML), 2012.

\bibitem{li2026end}
{\sc X.~Li, K.~Kan, D.~Verma, K.~Kumar, S.~Osher, and S.~Wu~Fung}, {\em
  End-to-end learning of safe optimal feedback control in high dimensions with
  control barrier function layers}, arXiv preprint arXiv:2607.20674,  (2026).

\bibitem{lin2021alternating}
{\sc A.~T. Lin, S.~Wu~Fung, W.~Li, L.~Nurbekyan, and S.~J. Osher}, {\em
  Alternating the population and control neural networks to solve
  high-dimensional stochastic mean-field games}, Proceedings of the National
  Academy of Sciences, 118 (2021), p.~e2024713118.

\bibitem{liu2023inverse}
{\sc H.~Liu, C.~Mou, and S.~Zhang}, {\em Inverse problems for mean field
  games}, Inverse Problems, 39 (2023), p.~085003.

\bibitem{liu2022inverse}
{\sc H.~Liu and S.~Zhang}, {\em On an inverse boundary problem for mean field
  games}, arXiv preprint arXiv:2212.09110,  (2022).

\bibitem{liu2021computational}
{\sc S.~Liu, M.~Jacobs, W.~Li, L.~Nurbekyan, and S.~J. Osher}, {\em
  Computational methods for first-order nonlocal mean field games with
  applications}, SIAM Journal on Numerical Analysis, 59 (2021), pp.~2639--2668.

\bibitem{liu2021splitting}
{\sc S.~Liu and L.~Nurbekyan}, {\em Splitting methods for a class of
  non-potential mean field games}, Journal of Dynamics and Games, 8 (2021),
  pp.~467--486.

\bibitem{mckenzie2024three}
{\sc D.~McKenzie, H.~Heaton, Q.~Li, S.~Wu~Fung, S.~Osher, and W.~Yin}, {\em
  Three-operator splitting for learning to predict equilibria in convex games},
  SIAM Journal on Mathematics of Data Science, 6 (2024), pp.~627--648.

\bibitem{mckenziedifferentiating}
{\sc D.~McKenzie, H.~Heaton, and S.~Wu~Fung}, {\em Differentiating through
  integer linear programs with quadratic regularization and {D}avis-{Y}in
  splitting}, Transactions on Machine Learning Research,  (2024).

\bibitem{meng2025recent}
{\sc T.~Meng, S.~Liu, S.~Wu~Fung, and S.~Osher}, {\em Recent advances in
  numerical solutions for {Hamilton--Jacobi} {PDEs}}, Communications on Applied
  Mathematics and Computation,  (2026),
  \url{https://doi.org/10.1007/s42967-026-00570-1}.
\newblock arXiv:2502.20833.

\bibitem{minty1962monotone}
{\sc G.~J. Minty}, {\em Monotone (nonlinear) operators in {H}ilbert space},
  Duke Mathematical Journal, 29 (1962), pp.~341--346.

\bibitem{onken2021neural}
{\sc D.~Onken, L.~Nurbekyan, X.~Li, S.~Wu~Fung, S.~Osher, and L.~Ruthotto},
  {\em A neural network approach applied to multi-agent optimal control}, in
  2021 European Control Conference (ECC), IEEE, 2021, pp.~1036--1041.

\bibitem{onken2022neural}
{\sc D.~Onken, L.~Nurbekyan, X.~Li, S.~Wu~Fung, S.~Osher, and L.~Ruthotto},
  {\em A neural network approach for high-dimensional optimal control applied
  to multiagent path finding}, IEEE Transactions on Control Systems Technology,
  31 (2023), pp.~235--251.

\bibitem{onken2021ot}
{\sc D.~Onken, S.~Wu~Fung, X.~Li, and L.~Ruthotto}, {\em {OT-Flow}: Fast and
  accurate continuous normalizing flows via optimal transport}, in Proceedings
  of the AAAI Conference on Artificial Intelligence, vol.~35, 2021,
  pp.~9223--9232.

\bibitem{papadakis2014optimal}
{\sc N.~Papadakis, G.~Peyr{\'e}, and E.~Oudet}, {\em Optimal transport with
  proximal splitting}, SIAM Journal on Imaging Sciences, 7 (2014),
  pp.~212--238.

\bibitem{park2026implicit}
{\sc Y.~Park, E.~Gelphman, S.~Osher, and S.~Wu~Fung}, {\em Implicit neural
  optimal transport via fixed-point optimization}, arXiv preprint
  arXiv:2605.10792,  (2026).

\bibitem{ren2024policy}
{\sc K.~Ren, N.~Soedjak, and S.~Tong}, {\em A policy iteration method for
  inverse mean field games}, Research in the Mathematical Sciences, 13 (2026),
  p.~34.

\bibitem{ren2024reconstructing}
{\sc K.~Ren, N.~Soedjak, K.~Wang, and H.~Zhai}, {\em Reconstructing a
  state-independent cost function in a mean-field game model}, Inverse
  Problems, 40 (2024), p.~105010,
  \url{https://doi.org/10.1088/1361-6420/ad7497}.

\bibitem{rockafellar1976monotone}
{\sc R.~T. Rockafellar}, {\em Monotone operators and the proximal point
  algorithm}, SIAM Journal on Control and Optimization, 14 (1976),
  pp.~877--898.

\bibitem{rockafellar1989proto}
{\sc R.~T. Rockafellar}, {\em Proto-differentiability of set-valued mappings
  and its applications in optimization}, Annales de l'Institut Henri
  Poincar\'e, Analyse Non Lin\'eaire, 6 (1989), pp.~449--482.

\bibitem{rockafellarwets1998}
{\sc R.~T. Rockafellar and R.~J.-B. Wets}, {\em Variational Analysis}, vol.~317
  of Grundlehren der mathematischen Wissenschaften, Springer, Berlin, 1998.

\bibitem{ruthotto2020machine}
{\sc L.~Ruthotto, S.~J. Osher, W.~Li, L.~Nurbekyan, and S.~Wu~Fung}, {\em A
  machine learning framework for solving high-dimensional mean field game and
  mean field control problems}, Proceedings of the National Academy of
  Sciences, 117 (2020), pp.~9183--9193.

\bibitem{ryu2022large}
{\sc E.~K. Ryu and W.~Yin}, {\em Large-scale convex optimization: algorithms \&
  analyses via monotone operators}, Cambridge University Press, 2022.

\bibitem{shaban2019truncated}
{\sc A.~Shaban, C.-A. Cheng, N.~Hatch, and B.~Boots}, {\em Truncated
  back-propagation for bilevel optimization}, in The 22nd international
  conference on artificial intelligence and statistics, PMLR, 2019,
  pp.~1723--1732.

\bibitem{stuart2020inverse}
{\sc A.~M. Stuart and M.-T. Wolfram}, {\em Inverse optimal transport}, SIAM
  Journal on Applied Mathematics, 80 (2020), pp.~599--619.

\bibitem{vidal2025kernel}
{\sc A.~Vidal, S.~Wu~Fung, S.~Osher, L.~Tenorio, and L.~Nurbekyan}, {\em Kernel
  expansions for high-dimensional mean-field control with non-local
  interactions}, in 2025 American Control Conference (ACC), IEEE, 2025,
  pp.~4164--4171.

\bibitem{vidal2023taming}
{\sc A.~Vidal, S.~Wu~Fung, L.~Tenorio, S.~Osher, and L.~Nurbekyan}, {\em Taming
  hyperparameter tuning in continuous normalizing flows using the {JKO}
  scheme}, Scientific Reports, 13 (2023), p.~4501.

\bibitem{vu2013splitting}
{\sc B.~C. V{\~u}}, {\em A splitting algorithm for dual monotone inclusions
  involving cocoercive operators}, Advances in Computational Mathematics, 38
  (2013), pp.~667--681, \url{https://doi.org/10.1007/s10444-011-9254-8}.

\bibitem{wang2025primal}
{\sc X.~Wang, S.~Wu~Fung, and L.~Nurbekyan}, {\em A primal-dual
  price-optimization method for computing equilibrium prices in mean-field
  games models}, Dynamic Games and Applications,  (2025),
  \url{https://doi.org/10.1007/s13235-025-00688-6}.
\newblock Online first.

\bibitem{fung2026generalization}
{\sc S.~Wu~Fung and B.~Berkels}, {\em A generalization bound for a family of
  implicit networks}, Neurocomputing, 678 (2026), p.~133136.

\bibitem{fung2022jfb}
{\sc S.~Wu~Fung, H.~Heaton, Q.~Li, D.~McKenzie, S.~Osher, and W.~Yin}, {\em
  {JFB}: {J}acobian-free backpropagation for implicit networks}, in Proceedings
  of the AAAI Conference on Artificial Intelligence, vol.~36, 2022,
  pp.~6648--6656.

\bibitem{yinlearning}
{\sc W.~Yin, D.~McKenzie, and S.~Wu~Fung}, {\em Learning to optimize: Where
  deep learning meets optimization and inverse problems}, SIAM News,  (2022).

\bibitem{yu2025equilibrium}
{\sc J.~Yu, J.-G. Liu, and H.~Zhao}, {\em Equilibrium correction iteration for
  a class of mean-field game inverse problems}, Inverse Problems, 41 (2025),
  p.~125009, \url{https://doi.org/10.1088/1361-6420/ae2b09}.

\bibitem{yu2024bilevel}
{\sc J.~Yu, Q.~Xiao, T.~Chen, and R.~Lai}, {\em A bilevel optimization method
  for inverse mean-field games}, Inverse Problems, 40 (2024), p.~105016.

\bibitem{zhang2025surrogate}
{\sc J.~Zhang, X.~Yang, C.~Mou, and C.~Zhou}, {\em Learning surrogate potential
  mean field games via {G}aussian processes: A data-driven approach to
  ill-posed inverse problems}, Journal of Computational Physics, 543 (2025),
  p.~114412, \url{https://doi.org/10.1016/j.jcp.2025.114412}.

\end{thebibliography}

\end{document}